\documentclass{amsart}
\usepackage{latexsym,amssymb}
\usepackage{eucal}
\usepackage{mathrsfs}
\usepackage[all]{xy}
\usepackage{color}
\usepackage{tikz}
\usepackage[hypertex]{hyperref}
\newtheorem{thm}{Theorem}[section]
\newtheorem{lem}[thm]{Lemma}
\newtheorem{prop}[thm]{Proposition}
\newtheorem{cor}[thm]{Corollary}

\theoremstyle{definition}
\newtheorem{defn}[thm]{Definition}
\theoremstyle{remark}
\newtheorem{rmk}[thm]{Remark}
\newtheorem{ex}[thm]{Example}
\newtheorem{exs}[thm]{Examples}

\newtheorem{notn}[thm]{Notation}
\newtheorem{term}[thm]{Terminology}

\numberwithin{equation}{section}

\DeclareMathAlphabet{\mathbfit}{OT1}{cmr}{bx}{it}

\newcommand{\cat}[1]{\mathbf{{#1}}}

\newcommand{\op}[1]{\mathscr{#1}}

\newcommand\Om{\Omega}
\newcommand{\im}{\operatorname{Im}}
\newcommand{\ob}{\operatorname{Ob}}
\newcommand{\mor}{\operatorname{Mor}}

\newcommand\egal[2]{\overset {#1}{\underset {#2}\rightrightarrows }}
\newcommand\adjunct[2]{\overset {#1}{\underset {#2}\rightleftarrows }}
\newcommand{\ve}{\varepsilon}
\newcommand\ma{\mathcal M_{A}}
\newcommand\mav{\mathcal M_{A}^{V}}
\newcommand\ch{\cat {Ch}_{R}^{\geq 0}}

\newcommand{\K}{\mathbb K}
\newcommand{\T}{\mathbb T}

\newcommand{\coker}{\operatorname{coker}}

\newcommand{\id}{\text{Id}}

\newcommand{\W}{\mathsf {WE}}
\newcommand{\C}{\mathsf {C}}

\renewcommand\path{\operatorname{Path}}
\newcommand\si{s^{-1}}
\renewcommand\hom{\operatorname{Hom}}

\begin{document}

\title[The homotopy theory of coalgebras over a comonad]{The homotopy theory of coalgebras over a comonad}

\author{Kathryn Hess}
\author{Brooke Shipley}

\address{MATHGEOM \\
    \'Ecole Polytechnique F\'ed\'erale de Lausanne \\
    CH-1015 Lausanne \\
    Switzerland}
    \email{kathryn.hess@epfl.ch}

\thanks{This material is based upon work by the second author supported by the National Science Foundation under Grants No.~0706877 and No.~1104396. Both authors are deeply grateful to Prof.~Peter May and the University of Chicago for the invitations that enabled this collaboration.}
\address{Department of Mathematics, Statistics and Computer Science, University of Illinois at
Chicago, 508 SEO m/c 249,
851 S. Morgan Street,
Chicago, IL, 60607-7045, USA}
    \email{bshipley@math.uic.edu}

    \date{\today}

\begin{abstract}  Let $\K$ be a comonad on a model category $\cat M$. We provide conditions under which the associated category $\cat M_{\K}$ of $\K$-coalgebras admits a model category structure such that the forgetful functor $\cat M_{\K}\to \cat M$ creates both cofibrations and weak equivalences.  

We provide concrete examples that satisfy our conditions and are relevant in descent theory and in the theory of Hopf-Galois extensions.  These examples are specific instances of the following categories of comodules over a coring (co-ring).  
For any semihereditary commutative ring $R$, let $A$ be a dg $R$-algebra  that is homologically simply connected. Let $V$ be an $A$-coring that is semifree as a left $A$-module on a degreewise $R$-free, homologically simply connected graded module of finite type. We show that there is a model category structure on the category $\ma$ of right $A$-modules satisfying the conditions of our existence theorem with respect to the comonad $-\otimes _{A}V$ and conclude that the category $\mav$ of $V$-comodules in $\ma$ admits a model category structure of the desired type.   Finally, under extra conditions on $R$, $A$, and $V$, we describe fibrant replacements in $\mav$ in terms of a generalized cobar construction. 
\end{abstract}

\maketitle
 \tableofcontents

 \section{Introduction}
Let $\cat M$ be a model category, and let $\T$ be a monad acting on $\cat M$.  There are well known and very useful conditions under which it is possible to transfer the model category structure from $\cat M$ to the category $\cat M^{\T}$ of $\T$-algebras in $\cat M$ so that the forgetful functor $\cat M^{\T}\to \cat M$ creates both weak equivalences and fibrations \cite[Lemma 2.3]{schwede-shipley}.    In particular, the category $\cat M$ should be cofibrantly generated for the results of \cite{schwede-shipley} to be applicable. 

Let $\K$ be a comonad acting on $\cat M$. Dualizing the hypotheses of Lemma 2.3 in \cite {schwede-shipley} does not provide  realistic conditions under which to transfer the model category structure from $\cat M$ to the category $\cat M_{\K}$ of $\K$-coalgebras, primarily because ``cosmall'' objects, and thus fibrantly generated model categories, are rare.   To avoid this problem, we instead generalize \cite [Section 2]{hess:hhg} and take an approach 
that  is inspired by the construction of factorizations and liftings by induction on degree, which is familiar to practitioners of rational and algebraic homotopy theory.  As long as the class of weak equivalences in $\cat M$ admits a filtration by ``$n$-equivalences'' that are compatible in a reasonable way with the comonad $\K$ (cf.~axioms (WE1) and (K0)-(K6) in section 5), we can guarantee the existence of a model category structure on $\cat M_{\K}$ such that the forgetful functor $\cat M_{\K}\to \cat M$ creates both weak equivalences and cofibrations. One advantage to our approach is that it enables us, under reasonable conditions, to describe fibrant objects and fibrations explicitly, rather than simply characterizing them in terms of the right lifting property.

Our main theorem (Theorem \ref{thm:bigthm}) could certainly easily be dualized, giving rise to an existence theorem for  model category structure on $\cat M^{\T}$ such that the forgetful functor $\cat M^{\T}\to \cat M$ creates both weak equivalences and fibrations, for nice enough monads $\T$, even if $\cat M$ is not cofibrantly generated.   Such a theorem would be complementary  to the results in \cite{schwede-shipley}.

\subsection*{Organization of the paper} We sketch the basic theory of comonads and their coalgebras  in the next section of this paper.  In particular we recall conditions under which categories of coalgebras are complete (Propositions \ref{prop:barr-wells} and \ref{prop:adamek}).  Since our main theorem is easiest to apply when the underlying model category is \emph{injective}, i.e., when its cofibrations are exactly the monomorphisms, we devote section 3 to  proving an existence result for injective model category structures (Theorem \ref{thm:icm}), which we then apply to showing that, for any commutative ring $R$, the category $\ma$ of right modules over any differential graded (dg) $R$-algebra $A$ admits an injective model category structure, in which the weak equivalences are the quasi-isomorphisms (Proposition \ref{prop:moda}).   

In section 4 we recall from \cite{hess:hhg} the notion of a Postnikov presentation of a model category and the related general existence theorem for model category structures in which the cofibrations and weak equivalences are created by a left adjoint (Theorem \ref{thm:postnikov}).  We can then state and prove our main theorem (Theorem \ref{thm:bigthm}) in section 5, providing conditions on $\cat M$ and $\K$ under which the category $\cat M_{\K}$ of $\K$-coalgebras admits a model category structure such that the forgetful functor $\cat M_{\K}\to \cat M$ creates both cofibrations and weak equivalences.   We show, moreover, that if $\cat M$ satisfies a certain ``Blakers-Massey-type'' condition, and its class of weak equivalences verifies two reasonable extra conditions, then the existence theorem for model category structure on $\cat M_{\K}$ holds under conditions on $\K$ that are somewhat easier to check (Proposition \ref{prop:K6prime}).  

In the last two sections of the paper, we apply our existence theorem to a concrete class of examples that is relevant to both descent theory \cite{hess:descent} and the theory of Hopf-Galois extensions \cite{hess:hhg}.  Let $R$ be a semihereditary commutative ring, let $A$ be a dg $R$-algebra, and let $V$ be an $A$-coring, i.e., a comonoid in the category of $A$-bimodules.  We show that if $A$ and $R\otimes _{A}V$ are both homologically simply connected,  and $V$ is $A$-semifree on an $R$-free graded module of finite type, then the category $\mav$ of right $V$-comodules in the category of right $A$-modules admits a model category structure such that the forgetful functor $\mav \to \ma$ creates both cofibrations and weak equivalences (Theorem \ref{thm:mav}).  Under further conditions on $R$, $A$ and $V$, we prove that fibrant replacements in $\mav$ can be built using certain generalized cobar constructions (Theorem \ref{thm:fib-repl}).

It is worth noting that while the proof of the existence of model category structure on $\mav$ requires that the  \emph{left} $A$-module structure of $V$ satisfy certain properties, we need to impose conditions on the \emph{right} $A$-module structure of $V$ in order to construct nice fibrant replacements.

In an upcoming article \cite{hess-shipley2}, we will provide further classes of explicit applications of Theorem \ref{thm:bigthm}, including categories of comodules over comonoids in pointed simplicial sets and categories of comodule spectra over suspension spectra.  These cases are much harder to study, as the underlying categories are neither cartesian nor additive.
 
\subsection*{Related work}
In \cite {quillen} Quillen established the first model category structure on a particular category of coalgebras over a comonad,  the category of 1-connected, cocommutative dg coalgebras over $\mathbb Q$. Almost thirty years later, in \cite [Theorem 7.6]{blanc} Blanc provided conditions, complementary to those given here, under which a ``right'' model category structure could be transfered from an underlying model category to a category of coalgebras.
In an unpublished paper from the same period \cite{getzler-goerss}, Getzler and Goerss proved the existence of a model category structure on the category of dg coalgebras over a field. Hinich then generalized Quillen's work,  defining a simplicial model category structure on the category of unbounded cocommutative coalgebras over a field of characteristic zero, but where the class of weak equivalences was strictly smaller than that of quasi-isomorphisms  \cite {hinich}.

 In 2003 Aubry and Chataur proved the existence of  model category structures on (certain) cooperads and coalgebras over them in unbounded chain complexes over a field \cite{aubry-chataur}. Smith established results along the same lines in \cite{smith} in 2011.  In 2010, Stanculescu used the dual of the Quillen path-object argument to establish a model structure on comonoids given a functorial cylinder object for comonoids \cite{stanculescu}.  In 2009, the first author showed that in a Cartesian model category, such as topological spaces, simplicial sets, or small categories, the category of comodules inherits a model structure from the underlying category because the category of comodules is equivalent to a slice (or over) category~\cite[1.2.1]{hess:hhg}
 
In his 2003 thesis  \cite{lefevre}, Lef\`evre showed that for any twisting cochain $\tau:C\to A$ such that twisted tensor product $C\otimes_{\tau}A$ is acyclic, there is a  model category structure on the category of  unbounded, coconnected $C$-comodules such that the functor  $\cat{Comod}_{C}\to \cat {Mod}_{A}$ induced by $\tau$ creates weak equivalences and cofibrations.  Finally, Positselski recently published a book \cite{positselski} in which he defined a model category structure on the category of comodules over a curved dg coalgebra over a field, in which the class of weak equivalences is strictly stronger than that of quasi-isomorphisms.

 \subsection*{Notation and conventions}

\begin{itemize}
\item Let $\cat C$ be a small category, and let $A,B\in \ob \cat C$.  In these notes, the set of morphisms from $A$ to $B$ is denoted $\cat C(A,B)$.  The identity morphism on an object $A$ will often be denoted $A$ as well.
\item A terminal (respectively, initial) object in a category is denoted $e$ (respectively, $\emptyset$).
\item If $L:\cat C \adjunct{}{} \cat D: R$ are adjoint functors, then we denote the natural bijections
$$\cat C(C, RD) \xrightarrow \cong \cat D(LC, D): f \mapsto f^\flat$$
and 
$$ \cat D(LC, D)\xrightarrow \cong \cat C(C, RD): g \mapsto g^\sharp$$
for all objects $C$ in $\cat C$ and $D$ in $\cat D$.
\end{itemize}
 \section{Comonads and their coalgebras}
 
 \begin{defn} Let $\cat D$ be a category.  A \emph{comonad} on $\cat D$ consists of an endofunctor $K:\cat D\to \cat D$, together with natural transformations $\Delta:K\to K\circ K$ and $\ve:K\to Id_{\cat C}$ such that $\Delta$ is appropriately coassociative and counital, i.e., $\K=(K,\Delta, \ve)$ is a comonoid in the category of endofunctors of $\cat D$.
\end{defn}

\begin{ex}  If $L:\cat C \adjunct{}{} \cat D:R$ is a pair of adjoint functors, with unit $\eta:Id_{\cat C}\to RL$ and counit $\ve: LR\to Id_{\cat D}$, then  $(LR, L\eta R, \ve)$ is a comonad on $\cat D$.
\end{ex}

There is a category of ``coalgebras'' associated to any comonad.

\begin{defn}\label{defn:Kcoalg}  Let $\K=(K, \Delta, \ve)$ be a comonad on $\cat D$.  The objects of the \emph{Eilenberg-Moore category of $\K$-coalgebras}, denoted $\cat D_{\K}$,  are pairs $(D, \delta)$, where $D\in \ob \cat D$ and $\delta\in \cat D(D, KD)$, which is appropriately coassociative and counital, i.e.,
$$K\delta \circ \delta = \Delta_{D}\circ \delta\quad\text{and}\quad \ve_{D}\circ \delta =Id_{D}.$$ 
A morphism in $\cat D_{\K}$ from $(D,\delta)$ to $(D',\delta')$ is a morphism $f:D\to D'$ in $\cat D$ such that $Kf\circ \delta=\delta'\circ f$.
\end{defn}

The category $\cat D_{\K}$ of $\K$-coalgebras is related to the underlying category $\cat D$ as follows.

\begin{rmk} \label{rmk:K-adjunct} Let $\K=(K, \Delta, \ve)$ be a comonad on $\cat D$. The forgetful functor $$U_{\K}:\cat D_{\K}\to \cat D$$ admits a right adjoint
$$F_{\K}:\cat D\to \cat D_{\K},$$
called the \emph{cofree $\K$-coalgebra functor}, defined on objects by
$F_{\K}(X) = (KX, \Delta_{X})$
and on morphisms by
$F_{\K}(f)=Kf.$
Note that $\K$ itself is the comonad associated to the $(U_{\K},F_{\K})$-adjunction.\end{rmk}

Since our goal is to establish a model category structure on $\cat D_{\K}$ when $\cat D$ is itself a model category, we need to recall how limits of $\K$-coalgebras are constructed. Colimits pose no problem, as they are created by the forgetful functor. 

We begin with an important special case of limits.

\begin{lem}\label{lem:barr-wells}\cite{barr-wells}  Let $\K=(K, \Delta, \ve)$ be a comonad on $\cat D$. Any $\K$-coalgebra $(D,\delta)$ is the equalizer in $\cat D_{\K}$ of the diagram
$$F_{\K}D \egal{K\delta}{\Delta_{D}} F_{\K}(KD).$$
\end{lem} 

Under the following condition on the functor underlying the comonad $\K$, the category of $\K$-coalgebras actually admits all equalizers.

\begin{defn} Let $\cat J$ denote the category with $\ob \cat J=\mathbb N$ and 
$$\cat J(n,m)=\begin{cases} \{j_{n,m}\}:&n\geq m\\ \emptyset:& n<m,\end{cases}$$
where $j_{m,m}=\id _{m}$ for all $m$.

A functor $F:\cat C\to \cat D$ \emph{preserves limits of countable
chains} if there is a natural isomorphism
$$\tau: F\circ \lim_{n}\Rightarrow \lim _{n}\circ F^{\cat J}$$
of functors from the diagram category $\cat C^{\cat J}$ to $\cat D$.
\end{defn}

\begin{prop}\label{prop:barr-wells} \cite{barr-wells} Let $\K=(K, \Delta, \ve)$ be a comonad on  a complete category $\cat D$.  If $K$ commutes with countable inverse limits, then $\cat D_{\K}$ admits all equalizers and is therefore complete.
\end{prop} 

\begin{proof} Barr and Wells prove the dual result for coequalizers of algebras over a monad in \cite{barr-wells}.  To give the reader some intuition for the nature of limits in $\cat D_{\K}$, we provide a sketch of the dual to the proof in \cite{barr-wells}.  

Let $(C,\gamma)\egal {f}{g} (D,\delta)$ be a diagram in $\cat D_{\K}$. Consider the following diagram in $\cat D$.
$$\xymatrix{C\ar [rr]^\gamma &&KC \ar @<0.3ex>[rr]^{K\gamma}\ar @<-0.3ex>[rr]_{\Delta_{C}}&&K^2C\\
B_{0}\ar[u]^{b_{0}}&&KB_{0}\ar [u]^{Kb_{0}}\ar [ll]_{\ve_{B_{0}}}\ar @<0.3ex>[urr]^(0.35){K(\gamma b_{0})}\ar @<-0.3ex>[urr]_(0.35){\Delta _{C}Kb_{0}}&&K^2B_{0}\ar [u]^{K^2b_{0}}\\
B_{1}\ar[urr]^{\beta _{1}}\ar[u]^{b_{1}}&&KB_{1}\ar [u]^{Kb_{1}}\ar [ll]_{\ve_{B_{1}}}\ar @<0.3ex>[urr]^(0.35){K\beta_{1}}\ar @<-0.3ex>[urr]_(0.35){\Delta _{B_{0}}Kb_{1}}&&K^2B_{1}\ar [u]^{K^2b_{1}}\\
B_{2}\ar[urr]^{\beta _{2}}\ar[u]^{b_{2}}&&KB_{2}\ar [u]^{Kb_{2}}\ar [ll]_{\ve_{B_{2}}}\ar @<0.3ex>[urr]^(0.35){K\beta_{2}}\ar @<-0.3ex>[urr]_(0.35){\Delta _{B_{1}}Kb_{2}}&&K^2B_{2}\ar [u]^{K^2b_{2}}\\
{\vdots}\ar [u]&&{\vdots}\ar [u]&&\vdots\ar [u]
}$$
Here, $b_{0}:B_{0}\to C$ is the equalizer of $C\egal fg D$ in $\cat D$, while if $i>0$, then $B_{i}$ is the limit of the part of the diagram above it and into which it maps.  The morphisms $b_{i}$ and $\beta_{i}$ are the natural cone maps from the limit.

If $B=\lim_{i\geq 0}B_{i}$, and
$$\beta = \lim_{i\geq 1}\beta _{i}: B\to \lim _{i\geq 1}KB_{i-1}\cong KB,$$  
then $(B,\beta)$ is a $\K$-coalgebra, which equalizes $(C,\gamma)\egal {f}{g} (D,\delta)$.  For the details of the (dual) argument, we refer the reader to \cite{barr-wells}.
\end{proof}

\begin{rmk}\label{rmk:invlimcoalg}   Let  $\K=(K, \Delta, \ve)$ be a comonad on $\cat D$ such that $K$ commutes with countable inverse limits, via a natural isomorphism $\tau: K\circ \lim _{n}\Rightarrow \lim_{n}\circ K^{\cat J}$.  As is certainly well known to those familiar with comonads,
 the forgetful functor $U_{\K}$ then also commutes with countable inverse limits.  Indeed, if 
$$\cdots\xrightarrow {p_{n+2}} (C_{n+1},\gamma_{n+1})\xrightarrow {p_{n+1}} (C_{n},\gamma_{n})\xrightarrow {p_{n}}\cdots \xrightarrow {p_{1}}(C_{0},\gamma_{0})$$
is a tower of $\K$-coalgebra morphisms, then the morphism
$$(\gamma_{n})_{n\geq 0}: (C_{n})_{n\geq 0}\to (KC_{n})_{n\geq 0}$$
of towers in $\cat D$ induces a morphism in $\cat D$
$$\lim _{n}C_{n}\xrightarrow{\lim _{n}\gamma_{n}} \lim _{n}KC_{n}\underset \cong{\xrightarrow{\tau^{-1}}} K(\lim C_{n}),$$
which is a $\K$-coalgebra structure on $\lim_{n}C_{n}$. Both coassociativity and counitality follow from the universal property of the limit and the naturality of $\tau$, which together imply that 
$$\lim _{n}\Delta_{C_{n}}\circ \tau=\tau\circ\tau\circ \Delta_{\lim _{n}C_{n}}: K(\lim_{n}C_{n})\to \lim_{n} K^2C_{n}$$ and 
$$\lim _{n}\ve_{C_{n}}\circ \tau= \ve_{\lim _{n}C_{n}}: K(\lim_{n}C_{n})\to \lim_{n}C_{n}.$$
It follows that $\lim_{n}(C_{n},\gamma_{n})=(\lim_{n}C_{n},\tau^{-1}\circ \lim_{n}\gamma_{n})$.
\end{rmk}

Once we know how to construct equalizers of $\K$-coalgebra morphisms, we can easily describe products and pullbacks, using the formulas of  the next lemma.

\begin{lem} Let $\K=(K, \Delta, \ve)$ be a comonad on $\cat D$. 
\begin{enumerate}
\item Products of cofree $\K$-coalgebras exist.  In particular, $$F_{\K}X\times F_{\K}Y\cong F_{\K}(X\times Y)$$ for all $X,Y\in \ob \cat D$.
\item For any $\K$-coalgebra $(D,\delta)$, the product $(D,\delta)\times F_{\K}X$ is the equalizer of the diagram
$$F_{\K} D \times F_{\K}X \egal {K\delta \times \id}{\Delta _{D}\times \id} F_{\K}(KD)\times F_{\K}X,$$
if it exists.
\item For any morphism $f:X\to Y$ in $\cat D$ and any morphism $g: (D,\delta)\to F_{\K}Y$ of $\K$-coalgebras, the pullback of $F_{\K}f$ and $g$ is the equalizer of the diagram
$$(D,\delta )\times F_{\K}X\egal {F_{\K}f\circ p_{2}}{g\circ p_{1}}F_{\K}Y,$$
if it exists. Here $p_{1}:(D,\delta )\times F_{\K}X\to (D,\delta)$ and $p_{2}:(D,\delta )\times F_{\K}X\to F_{\K}X$ are the natural projection maps.
\end{enumerate}
\end{lem}

\begin {proof} (1) This isomorphism follows from the fact that $F_{\K}$ is a right adjoint.

(2) Since limits commute with limits, this formula for $(D,\delta)\times F_{\K}X$ is a consequence of Lemma \ref{lem:barr-wells}.

(3) This is the standard description of a pullback as an equalizer.
\end{proof}

Under a reasonable condition on the category on which a comonad $\K$ acts, the category of $\K$-coalgebras is complete under an even milder condition on $K$ than preservation of inverse limits.   Recall that a category is \emph{well-powered} if the subobjects of any object form a set, rather than a proper class.  Any locally presentable category is well-powered \cite{adamek-rosicky}.  Recall that a morphism $g:B\to C$ in any category $\cat C$ is a \emph{monomorphism} if for all pairs of morphisms $f,f':A\to B$ with target $B$,
$$gf=gf'\Longrightarrow f=f'.$$ 

\begin{prop} \cite {adamek}\label{prop:adamek}  Let $\K=(K, \Delta, \ve)$ be a comonad on a well-powered category $\cat D$.  If $K$ preserves monomorphisms, then $\cat D_{\K}$ is complete.
\end{prop}

Ad\'amek proves this proposition by providing an explicit ``solution set''-type  construction of an equalizer of $\K
$-coalgebras.

\section{Injective combinatorial model structures}

In this section we provide conditions under which  a model category admits a Quillen-equivalent \emph{injective} model category structure, i.e., a model category structure in which the cofibrations are exactly the monomorphisms.  The injectivity condition is important in this paper as it simplifies considerably the existence proof for model category structures on categories of coalgebras.

We then apply our existence theorem to establishing that categories of differential graded modules over differential graded algebras that are degreewise flat over the ground ring admit injective model category structures.  

\subsection{The existence theorem}
 We apply Smith's argument for constructing combinatorial model
categories to establish the existence of an injective model category structure. We follow Lurie's version of the argument  \cite[A.2.6.8]{lurie}, but see  also  \cite[1.7]{beke}, or \cite[4.3]{rosicky}.   

Let $\cat M$ be a category endowed with a ``standard'' combinatorial model (SCM) structure
(see Definition \ref{defn:scm} below).
In Theorem \ref{thm:icm} we establish the existence of an  injective combinatorial model  (ICM) structure on $\cat M$ with the same weak equivalences and cofibrations exactly the monomorphisms. 

There is an ICM structure on a category $\cat M$ only if the class of all monomorphisms in $\cat M$ is generated by a set.  To state conditions under which there is a such a generating set, we need the following standard notions.

\begin{defn}\label{defn:effunion} Let $\cat C$ be a category.  For every pair of monomorphisms 
$$A \xrightarrow a X \xleftarrow b B$$ 
with a common codomain, let 
$$A\cup B:= A \coprod_{A\underset X \times B}B,$$
the pushout of $A \leftarrow A\underset X\times B \rightarrow B$, where $A\underset X\times B$ is the pullback of $a$ and $b$.

The category $\cat C$ has \emph{effective unions} if each of the natural morphisms
$$\xymatrix{A \ar [r]& A\cup B \ar [d] & B\ar [l]\\ &X}$$
is a monomorphism, for every pair of monomorphisms $A \xrightarrow a X \xleftarrow b B$.
\end{defn}

\begin{defn}  If $\mathsf X$ is a set of morphisms in a category $\cat C$, then $\mathsf X$-inj is the class of morphisms in $\cat C$ satisfying the right lifting property with respect to $\mathsf X$, while $\mathsf X$-cof is the class of morphisms satisfying the left lifting property with respect to $\mathsf X$-inj. In other words,  a morphism $p:E\to B$  is in $\mathsf {X}$-inj if for any commuting diagram in $\cat C$
$$\xymatrix{ A\ar[r]^f\ar [d]^{i} & E\ar[d]^p\\ X\ar[r]^g& B,}$$
where $i\in \mathsf{X}$, there is a morphism $h:X\to E$ such that $ph =g$ and $ hi=f$, while a morphism $j:Y\to Z$ is in $\mathsf {X}$-cof if for any commuting diagram in $\cat C$
$$\xymatrix{ Y\ar[r]^f\ar [d]^{j} & E\ar[d]^p\\ Z\ar[r]^g& B,}$$
where $p\in \mathsf{X}$-inj, there is a morphism $h:Z\to E$ such that $ph =g$ and $ hj=f$
\end{defn}   

\begin{lem}\label{lem:beke}\cite [1.12]{beke}
Let $\cat C$ be a category. If 
\begin{enumerate}
\item $\cat C$ is locally presentable,
\item subobjects in $\cat C$ have effective unions, and
\item the class of  monomorphisms is closed under transfinite composition,
\end{enumerate}
then there is a set of monomorphisms $\mathsf I$ in $\cat C$ such that the class of  all monomorphisms is equal to $\mathsf I$-cof.  
\end{lem}

Recall that a
model structure is {\em combinatorial} if it is cofibrantly generated and the underlying category
is locally presentable.

\begin{defn}\label{defn:scm} 
A combinatorial model structure such that any cofibration is a monomorphism 
is a {\em standard combinatorial model}  (SCM) structure if the underlying
category $\cat M$ satisfies the hypotheses of Lemma \ref{lem:beke}. 
\end{defn}

We need one more definition before constructing the injective model structure on $\cat M$ .

\begin{defn}\cite[A.1.2.2]{lurie}\label{defn:weaksat}  A class of morphisms in a category is \emph{weakly saturated} if it is closed under pushouts, transfinite compositions and retracts.
\end{defn}
 
\begin{thm}\label{thm:icm}  Let $\cat M$ be a category with an SCM structure with weak equivalences $\mathsf W$.   Let $\mathsf C$ denote the class of monomorphisms in $\cat M$.  If $\mathsf W\cap \mathsf C$
is weakly saturated,  then there is a combinatorial model structure on $\cat M$ with weak equivalences $\mathsf W$ and cofibrations $\mathsf C$.   
\end{thm}  

\begin{term}  We refer to the model category stucture of the theorem above as the \emph{associated injective combinatorial model (ICM) category structure} on $\cat M$.
\end{term}

\begin{proof} We check the conditions from A.2.6.8 in \cite{lurie}. 

\begin{enumerate}
\item $\mathsf C$ is weakly saturated and generated by $\mathsf C_0$.

We take $\mathsf C_0$ to be the set of monomorphisms $\mathsf I$, the existence of which follows from Definition \ref{defn:scm} and Lemma \ref{lem:beke}.   Condition (1) then holds by definition since
$\mathsf C =  \mathsf I$-cof is weakly saturated by \cite[A.1.2.7]{lurie}.

\item  $\mathsf C\cap \mathsf W$ is weakly saturated.

This condition is the hypothesis of our theorem.

\item $\mathsf W$ is accessible.
           
This follows from \cite[4.1]{rosicky} or \cite[A.2.6.6]{lurie}: since the SCM structure on $M$ is combinatorial, $\mathsf W$ is accessible.
           
\item $\mathsf W$ satisfies the ``2 out of 3'' property.   

This is true because $\mathsf W$ is the set of weak equivalences of the original SCM structure on $\cat M$.

\item $ \mathsf C-{\text{inj}}\subseteq \mathsf W$.     
 
Let $\mathsf C_{s}$ be the cofibrations in the SCM structure on $\cat M$.  By definition
 $\mathsf C_s \subseteq \mathsf C$,  so
$$ \mathsf C-{\text{inj}}\subseteq \mathsf C_{s}-{\text{inj}}.$$ Since $\mathsf C_{s}-{\text{inj}}$
is the class of trivial fibrations in the original SCM structure on $\cat M$, it follows that
$ \mathsf C-{\text{inj}}\subseteq \mathsf W$.
\end{enumerate}

\end{proof}

\subsection{An ICM structure for dg modules}\label{subsec:dga}
 For any commutative ring $R$, let  $\ch$ denote the category of nonnegatively graded chain complexes of $R$-modules, endowed with its usual graded tensor product, which we denote simply $\otimes$. If $A$ is a monoid in $\ch$, i.e., a differential graded (dg) algebra, let $\ma$  denote the category of right $A$-modules.
 
 We begin by a few easy but useful observations concerning the categorical properties of $\ma$.
 
 \begin{lem}\label{lem:mono-mono}  A morphism in $\ma$ is a monomorphism if and only if the underlying morphism in $\ch$ is a monomorphism.
\end{lem}

\begin{proof}  Let $U:\ma \to \ch$ denote the forgetful functor. Let $f:M\to N$ be a morphism in $\ma$. It is obvious that if $U f$ is a monomorphism, then $f$ is as well.

If $Uf$ is not a monomorphism, then there exist $x,y: X\to UM$ in $\ch$ such that $x\not=y$ but $U f \circ x=U f \circ y: X\to UN$.  Taking transposes, we obtain
$$f\circ x^\flat = (Uf \circ x)^\flat=(U f \circ y)^\flat=f\circ y^\flat,$$
and thus $f$ is not a monomorphism, since $x^\flat \not= y^\flat$.
\end{proof}

\begin{lem}\label{lem:effunion-alg} The category $\ma$ has effective unions.
\end{lem}

\begin{proof}  Since pullbacks and pushouts in $\ma$ are created in $\ch$ and $\ch$ clearly has effective unions, this lemma is an immediate consequence of Lemma \ref{lem:mono-mono}. \end{proof}

\begin{lem}\label{lem:weaksat}  The class of monomorphisms in $\ma$ is closed under transfinite composition, and the class of monomorphisms in $\ma$ that are also quasi-isomorphisms is weakly saturated.
\end{lem}

\begin{proof} The transfinite composition of a sequence
$$M_{0}\hookrightarrow M_{1}\hookrightarrow \cdots \hookrightarrow M_{n}\hookrightarrow M_{n+1}\hookrightarrow \cdots$$
of monomorphisms of $A$-modules (seen, without loss of generality, as a sequence of inclusions) is simply the inclusion $M_{0}\hookrightarrow \bigcup _{n\geq 0}M_{n}$.  Transfinite compositions for larger ordinals are constructed similarly. The class of monomorphisms in $\ma$ is therefore closed under transfinite composition. 

Since homology commutes with filtered colimits, it follows that the transfinite composition of a sequence of monomorphisms that are quasi-isomorphisms is both a monomorphism and a quasi-isomorphism.  Furthermore the class of monomorphisms is closed under retracts for categorical reasons, and it is well known that the class of quasi-isomorphisms is as well.

Finally, since the cokernel of a monomorphism $j$ of chain complexes is acyclic if and only if $j$ is a quasi-isomorphism,  a pushout of a monomorphism that is a quasi-isomorphism is again a monomorphism and a quasi-isomorphism, as the cokernel of a pushout of $j$ along any morphism is isomorphic to $\coker j$.  
\end{proof}

\begin{prop}\label{prop:moda} For any dg $R$-algebra $A$, the category $\ma$ of right $A$-modules admits a combinatorial model category structure in which the cofibrations are the monomorphisms,  and the weak equivalences are the quasi-isomorphisms.  
\end{prop}

Note that this proposition implies, obviously, that $\ch$ itself admits a ICM structure.

\begin{proof}  There is a combinatorial model structure on $\ma$ obtained by right transfer of  the projective structure on $\ch$via the adjunction
$$\ch \adjunct{-\otimes A}{U} \ma,$$
as described in \cite{schwede-shipley}.  The fibrations in this model category structure are the chain maps that are surjections in positive degrees, and the weak equivalences are the quasi-isomorphisms. Let $\mathsf I$ denote the set of  generating cofibrations of the projective model structure on $\ch$. Recall that the class of cofibrations in this right-induced structure on $\ma$ can be constructed  by taking transfinite composition of pushouts of morphisms of the form $i\otimes A$ for $i\in \mathsf I$ and retractions of such. 

Recall moreover that $\mathsf I=\{S^{n}\hookrightarrow D^{n+1}\mid n\geq 0\}$, where $S^{n}=(R\cdot x_{n},0)$, the chain complex freely generated by exactly one generator of degree $n$, while $D^{n+1}=\big(R\cdot (x_{n},y_{n+1}),d\big )$, the chain complex freely generated by  one generator of degree $n$ and one of degree $n+1$, with $dy=x$.
If $i\in \mathsf I$, then $i\otimes A$ is a monomorphism of chain complexes, as the source and target of $i$ are degreewise $R$-free.  Since monomorphisms of chain complexes are preserved under pushout, transfinite composition and retraction, and colimits in $\ma$ are created in $\ch$,  the morphism of chain complexes underlying any cofibration in the induced model structure on $\ma$ is a monomorphism.  Lemma \ref{lem:mono-mono} therefore implies that every cofibration in the right-induced structure on $\ma$ is a monomorphism of right $A$-modules. 

The category $\ch$ is locally presentable~\cite[3.7]{shipley-hz}.  It follows that $\ma$ is also locally presentable, as $-\otimes A$ preserves filtered colimits \cite{adamek-rosicky}, \cite{gabriel-ulmer}.  

The existence of the desired model category structure on $\ma$ follows therefore from Lemma \ref{lem:weaksat} and Theorem \ref{thm:icm}. 
\end{proof}

\section{Left-induced model category structures}\label{sec:induced}

A common way of creating model structures is by transfer across adjunctions, such as the left-to-right adjunction applied in the proof of Proposition \ref{prop:moda}.   To construct model category structures on categories of coalgebras over a comonad, we need right-to-left transfer, as specified in the following definition.

\begin{defn}\label{defn:induction} 
Let $L: \cat C \to \cat M$ be a functor, where $\cat M$ is a model category. A model structure on $\cat C$ is \emph{left-induced} from $\cat M$ if $\mathsf {WE}_{\cat C}=L^{-1}(\mathsf {WE}_{\cat M})$ and $\mathsf {Cof}_{\cat C}=L^{-1}(\mathsf {Cof}_{\cat M})$.
\end{defn}

\begin{rmk} The terminology above is motivated by the fact that it is most natural to consider such model category structures when the functor $L$ is a left adjoint, such as the forgetful functor from the category of coalgebras over some comonad to the underlying category.
\end{rmk}

Before giving conditions under which left-induced structures exist, we introduce a bit of useful notation.

\begin{notn}Let $\mathsf X$ be any class of morphisms in a category $\cat C$. 
The closure of $\mathsf X$ under formation of retracts is denoted $\widehat{\mathsf X}$, i.e., 
$$f\in \widehat{\mathsf X}\Longleftrightarrow  \exists\; g\in \mathsf X \text{ such that $f$ is a retract of $g$} .$$
\end{notn}

\begin{defn} \label{defn:postnikov} Let $\mathsf X$ be a class of morphisms in a category $\cat C$ that is closed under pullbacks.  If $\lambda$ is an ordinal, and  $Y:\lambda ^{op}\to \cat C$ is a functor such that for all $\beta <\lambda$, the morphism $Y_{\beta +1}\to Y_{\beta}$ fits into a pullback
$$\xymatrix{Y_{\beta +1}\ar [d]_{}\ar [r]^{}&X'_{\beta+1}\ar [d]^{x_{\beta+1}}\\ Y_{\beta}\ar [r]^{k_{\beta}}&X_{\beta+1}}$$
for some $x_{\beta +1}: X'_{\beta +1}\to X_{\beta+1}$ in $\mathsf X$ and  $k_{\beta}:Y_{\beta}\to X_{\beta+1}$ in $\cat C$, while
$Y_{\gamma}:=\lim _{\beta<\gamma}Y_{\beta}$ for all limit ordinals $\gamma<\lambda$,
then the composition of the tower $Y$
$$\lim_{\lambda^{op}}Y_{\beta}\to Y_{0},$$
\emph{if it exists}, is an \emph{$\mathsf X$-Postnikov tower}.  

A \emph{Postnikov presentation} of a model category $(\cat M, \mathsf{Fib}, \mathsf{Cof}, \mathsf {WE})$ is  a pair of  classes of morphisms $\mathsf X$ and $\mathsf Z$ satisfying 
$$\mathsf {Fib}=\widehat{\mathsf {Post}_{\mathsf X}}\quad\text { and }\quad \mathsf {Fib}\cap \mathsf {WE}=\widehat{\mathsf {Post}_{\mathsf Z}}$$ 
and such that for all $f\in \mor \cat M$,
there exist 
\begin{enumerate}
\item [(a)] $i\in \mathsf{Cof} $ and $p\in \mathsf {Post}_{\mathsf Z}$ such
that $f=pi$; 
\item [(b)] $j\in \mathsf{Cof}\cap \mathsf{WE}$ and $ q\in \mathsf {Post}_{\mathsf X}$ such that
$f=qj$.
\end{enumerate}
\end{defn}

\begin{rmk}\label{rmk:post-closure} For any $\mathsf X$, the class  $\mathsf{Post}_{\mathsf X}$ is closed under pullbacks, since inverse limits commute with pullbacks.  Furthermore, $\mathsf{Post}_{\mathsf X}$ is clearly closed under composition of towers  as well.
\end{rmk}

\begin{rmk}\label{rmk:postnikov-lift} Let $\mathsf X$ and $\mathsf Y$ be two classes of morphisms in  a category $\cat C$ admitting pullbacks and inverse limits. If $\mathsf X\subset \mathsf{Y}$-inj, then $\mathsf {Post}_{\mathsf X}\subset \mathsf{Y}$-inj as well, and therefore $\widehat{\mathsf {Post}_{\mathsf X}}\subset \mathsf{Y}$-inj. In particular, for any model category $(\cat M, \mathsf{Fib}, \mathsf{Cof}, \mathsf {WE})$, the pair $(\mathsf {Fib}, \mathsf {Fib}\cap \mathsf {WE})$ is  a Postnikov presentation, which we call the \emph{generic Postnikov presentation} of $\cat M$.
\end{rmk}

The following right-to-left transfer theorem for Postnikov model structures was proved in \cite{hess:hhg}.

\begin{thm} \label{thm:postnikov} Let $(\cat M, \mathsf{Fib}, \mathsf{Cof}, \mathsf {WE})$ be a model category with Postnikov presentation $(\mathsf X, \mathsf Z)$.
Let $\cat C$ be a bicomplete category (i.e., $\cat C$ admits all limits and colimits), and let $L:\cat C\adjunct {}{}\cat M:R$ be an adjoint pair of functors.  Let 
$$\mathsf W=L^{-1}(\mathsf{WE}) \text { and } \mathsf C=L^{-1}(\mathsf{Cof}).$$   
If 
\begin{enumerate}
\item [(a)]$\mathsf{Post}_{R(\mathsf Z)}\subset \mathsf W$,
\end{enumerate}
 and for all $f\in \mor \cat C$
there exist 
\begin{enumerate}
\item [(b)] $i\in \mathsf{C} $ and $p\in \mathsf{Post}_{R( \mathsf Z)}$ such
that $f=pi$, and
\smallskip
\item [(c)]$j\in \mathsf{C}\cap \mathsf{W}$ and $q\in \mathsf{Post}_{R(\mathsf X)}$ such that
$f=qj$,
\end{enumerate}
then    $\mathsf W$, $\mathsf C$  and $\widehat{\mathsf{Post}}_{R(\mathsf X)}$ are the weak equivalences, cofibrations and fibrations in a model category structure on $\cat C$, with respect to which $L:\cat C \adjunct {}{}Ê\cat M: R$ is a Quillen pair.
\end{thm}

\section{Postnikov presentations and coalgebras}

Let $\K=(K, \Delta, \ve)$ be a comonad on a model category $(\cat M, \mathsf{Fib}, \mathsf{Cof}, \mathsf {WE})$. In this section we apply Theorem \ref{thm:postnikov} to provide conditions on $\K$ and $\cat M$  that guarantee that the associated category of coalgebras $\cat M_{\K}$ inherits a left-induced model category structure from $\cat M$.  

Our proofs are inductive and require the following sort of filtered structure on $\cat M$.  Note that, to simplify notation, we henceforth often suppress explicit mention of the distinguished classes of morphisms $(\mathsf{Fib}, \mathsf{Cof}, \mathsf {WE})$ when we refer to the model category $(\cat M, \mathsf{Fib}, \mathsf{Cof}, \mathsf {WE})$.

\begin{defn}\label{defn:filt-we}  The model category $\cat M$  has \emph{filtered weak equivalences} if it is endowed with a decreasing filtration
$$\W\subseteq  ...\subseteq \W_{n+1}\subseteq \W_{n}\subseteq ...\subseteq \W_{-1}=\mor \cat M$$ satisfying the following axiom.
\begin{description}
\item [(WE1)] For all $n$,  $\W_{n}$ is closed under composition.  If $f \in \W_{n}$ for all $n$, then $f$ is in $\W$.  
Moreover, if $f:A\to B$ and $g:B\to C$ are composable morphisms, then
\begin {itemize}
\item $f, gf\in \W_{n}\Longrightarrow g\in \W_{n}$, 
\item $g, gf\in \W_{n}\Longrightarrow f\in \W_{n-1}$, and
\item $gf\in \W_{n}$ and $g\in \W\Longrightarrow f\in \W_{n}$.
\end{itemize}
\end{description}

We refer to the morphisms in $\W_{n}$ as \emph{$n$-equivalences} and denote an $n$-equivalence by $ {\sim_{n}}$.  An object $X$ in $\cat M$ is called \emph{$(n-1)$-connected} if the unique morphism from $X$ to a terminal object is an $n$-equivalence.
\end{defn}

The comonads we consider satisfy the following compatibility with the  model category structure on $\cat M$. 

\begin{defn}\label{defn:tractable} Let  $\cat M$ be a model category with Postnikov presentation $(\mathsf X, \mathsf Z)$.  
A comonad $\K$ on $\cat M$ is \emph{tractable} if  the following axioms hold.
\begin{description}
\item[(K0)] $\cat M_{\K}$ is complete.
\item[(K1)] $\delta: D\to KD\in \mathsf {Cof}$ for all $\K$-coalgebras $(D,\delta)$.
\item [(K2)] $K$ preserves cofibrations.
\item [(K3)] For all $i: (C,\gamma)\to F_{\K}X$ in $U_{\K}^{-1}(\mathsf{Cof})$ and all $g: (C,\gamma)\to (D,\delta)$ in $\cat M_{\K}$, the induced morphism of $\K$-coalgebras
$$(i,g): (C,\gamma) \to F_{\K}X \times (D,\delta)$$
is also in $U_{\K}^{-1}(\mathsf{Cof})$, if the product exists.
\item [(K4)] For all $q:E\to B$ in $\mathsf Z$ and for all  morphisms $f:(D,\delta) \to F_{\K}B$ in $\cat M_{\K}$, the induced morphism in $\cat M$
$$U_{\K}\big( (D,\delta) \times_{F_{\K}B}F_{\K}E\big) \to U_{\K}\big(D,\delta)$$
is in $\W$, if the pullback exists in $\cat M_{\K}$. 
\end{description}
\end{defn}

When $\cat M$ is endowed with an injective model structure, there is a simple condition under which axioms (K0) through (K3) hold.

\begin{lem}\label{lem:1st-factor} Let  $\K$ be a comonad on a well-powered model category $\cat M$ with injective model category structure.  Axioms (K0) through (K3) hold for $\K$ if and only if the underlying functor $K$ preserves monomorphisms.
\end{lem}

\begin{rmk}  Many interesting comonads preserve monomorphisms.  We consider such an example, when $\cat M$ is a category of differential graded modules over a differential graded algebra, in the last two sections of this paper.  In \cite{hess-shipley2} we will treat examples of such comonads, when the underlying category is that of either pointed simplicial sets or Bousfield-Friedlander spectra.
\end{rmk}

\begin{proof}  Since the model category structure on $\cat M$ is injective, axiom (K1)  holds for all comonads $\K$, as every coalgebra structure map $\delta:D\to KD$ admits a retraction $\ve_{D}:KD\to D$ and is therefore a monomorphism.  Injectivity of the model category structure also implies that the functor $K$ preserves monomorphisms if and only if axiom (K2) is satisfied.

If $K$ preserves monomorphisms, then it follows from Proposition \ref{prop:adamek} that (K0) holds, while axiom (K3)  is a special case of the following result. Let $L:\cat C \adjunct {}{}\cat D:R$ be an adjoint pair of functors, and let $b:A\to B$ and $c:A\to C$ be morphisms in $\cat C$, inducing $(b,c): A\to B\times C$.  We claim that if $L(b)$ is a monomorphism, then $L(b,c)$ is as well. 

If $d,e:D\to L(A)$ are morphisms in $\cat D$ such that $L(b,c)\circ d =L(b,c)\circ e$, then 
$$L(b)\circ d =L(pr_{B})\circ L(b,c)\circ d = L(pr_{B})\circ L(b,c)\circ e= L(b)\circ e,$$
whence $d=e$, since $L(b)$ is a monomorphism.  We conclude that $L(b,c)$ is also a monomorphism.
\end{proof}

\begin{rmk} 
Let $L:\cat C \adjunct {}{}\cat D:R$ be an adjoint pair of functors.  If $L$ is faithful, then 
$$L^{-1}(\mathsf{Mono}_{\cat D})\subset \mathsf{Mono}_{\cat C}.$$
Indeed, if $f:A\to B$ is a morphism in $\cat C$ such that $L(f)$ is a monomorphism, and $g,h:C\to A$ are morphisms in $\cat C$ such that $fg=fh$, then $L(f)L(g)=L(f)L(h)$, whence $L(g)=L(h)$, as $L(f)$ is a monomorphism.  Since $L$ is faithful, we can conclude that $g=h$ and therefore that $f$ is a monomorphism.

It follows that if $\cat M$ is an injective model category, and $\K=(K,\Delta, \ve)$ is a comonad on $\cat M$ such that $K$ preserves monomorphisms, then every element of $U_{\K}^{-1}(\mathsf{Cof})$ is a monomorphism of $\K$-coalgebras, since $U_{\K}:\cat M_{\K}\to \cat M$ is faithful for all comonads $\K$.
\end{rmk}

To construct one type of Postnikov factorization in the category of coalgebras over a comonad $\K$, we make inductive arguments based on the following sort of compatibility between $\K$ and extra structure on the model category on which it acts. Moreover, in  order for condition (a) of Theorem \ref{thm:postnikov} to hold for the cofree $\K$-coalgebra adjunction,  certain towers  should satisfy a Mittag-Leffler-type condition.

\begin{defn}\label{defn:inductive} A  comonad $\K$ on a model category $\cat M$ that has a Postnikov presentation $(\mathsf X, \mathsf Z)$ and filtered weak equivalences \emph{allows inductive arguments}  if (K0) and the axioms below hold, where $\mathsf W_n=U_\K^{-1}(\W_n)$ and $\C=U_{\K}^{-1}(\mathsf{Cof})$.
\begin{description}
\item [(K5)] There is some $k$ such that the composition of any tower of countable length in $\mathsf{Post}_{F_{\K} \mathsf X}\cap \mathsf W_n$ is in $\mathsf W_{n-k}$, for all $n\geq k-1$.
\item [(K6)] For all $n\geq -1$ and for all $i: (C,\gamma)\to (D, \delta)\in \C \cap \mathsf{W}_{n}$, the induced morphism 
$$\big((i''u)^\sharp,i\big):   (C,\gamma) \to F_{\K}Q \times _{F_{K}P}(D,\delta)$$
is in $\mathsf W_{n+1}$, where 
$$\xymatrix{C\ar [d]_{i}\ar[r]^{u}&e\ar[d]^{i'}\\ D\ar [r]&P}$$
is a pushout in $\cat M$, and
$$\xymatrix{e\;\ar [rr]^{i'}\ar [dr]^{i''}_{\sim}&& P\\&Q\ar [ur]^q}$$
is a factorization with $i''\in \mathsf {Cof}\cap \W$ and $q\in \mathsf {Post}_{\mathsf X}$.
\end{description}
\end{defn}

\begin{rmk}\label{rmk:K5'} Axiom (K5) can sometimes be replaced by an axiom that should be easier to check. Let $\K$ be a comonad on $\cat M$ such that inverse limits and pullbacks in $\cat M_{\K}$ are created in $\cat M$ and such that the following axiom holds.
\begin{description}
\item [(K5')]   There is some $k$ such that the composition of any tower of morphisms in $\mathsf{Post}_{K(\mathsf{X})}\cap \W_{n}$ is  in $\W_{n-k}$ for all $n\geq k-1$.
\end{description}
Then $\K$ satisfies axiom (K5), since the fact that pullbacks and inverse limits of $\K$-coalgebras are created in $\cat M$ implies that $U_{\K}(\mathsf{Post}_{F_{\K}(\mathsf{X})}\cap \mathsf {W}_{n})\subseteq \mathsf{Post}_{K(\mathsf{X})}\cap \W_{n}$.
\end{rmk}

In last two sections of this paper we consider examples of tractable comonads that allow inductive arguments; we will treat further examples in \cite{hess-shipley2}.

Our goal in this section is to apply Theorem \ref{thm:postnikov} to proving the following existence result.

\begin{thm}\label{thm:bigthm}  Let $\cat M$ be a model category with filtered weak equivalences and  Postnikov presentation $(\mathsf X, \mathsf Z)$ such that $\mathsf Z\subseteq \mathsf {Post}_{\mathsf X}$.

If $\K$ is a tractable comonad on $\cat M$ that allows inductive arguments, then the category of $\K$-coalgebras, $\cat M_{\K}$, admits a model category structure such that
$$\mathsf {Cof}_{\cat M_{\K}}=U_{\K}^{-1}(\mathsf {Cof}) ,\quad \mathsf {WE}_{\cat M_{\K}}=U_{\K}^{-1}(\mathsf {WE})\quad \text{ and }\quad \mathsf {Fib}_{\cat M_{\K}}=\widehat {\mathsf {Post}_{F_{\K}\mathsf X}}.$$
\end{thm}

\begin{rmk} If $(\mathsf X, \mathsf Z)=(\mathsf{Fib}, \mathsf{Fib}\cap \W)$, the generic Postnikov presentation, then it is trivially true that $\mathsf Z\subseteq \mathsf {Post}_{\mathsf X}$.
\end{rmk}

We begin the proof of Theorem \ref{thm:bigthm} with the relatively simple observation that condition (a) of Theorem \ref{thm:postnikov} is satisfied under the hypotheses above.

\begin{prop}\label{prop:cond-a} Let $\cat M$ be a model category with filtered equivalences and a Postnikov presentation $(\mathsf X, \mathsf Z)$ such that $\mathsf Z\subseteq \mathsf{Post}_{\mathsf X}$.  If $\K$ is a comonad on $\cat M$ such that axioms (K0), (K4) and (K5) hold,  then $U_{\K}(\mathsf{Post}_{F_{\K}\mathsf Z})\subseteq \W$.
\end{prop}

\begin{proof} 
Since $\mathsf Z\subseteq \mathsf{Post}_{\mathsf X}$,
$$F_{\K}\mathsf Z\subseteq F_{\K}{\mathsf{Post}_{\mathsf X}} \subseteq {\mathsf{Post}_{F_{\K}\mathsf X}}$$ 
because $F_{\K}$ commutes with limits.   As ${\mathsf{Post}_{F_{\K}\mathsf X}}$ is closed under pullbacks and composition of towers (Remark \ref{rmk:post-closure}), it follows that ${\mathsf{Post}_{F_{\K}\mathsf Z}}\subseteq {\mathsf{Post}_{F_{\K}\mathsf X}}$. On the other hand, axiom (K4) implies that any morphism in ${\mathsf{Post}_{F_{\K}\mathsf Z}}$ is the composition of a tower of weak equivalences.  By axiom (K5), the composition of any tower in $\mathsf{Post}_{F_{\K} \mathsf X}\cap \mathsf W$ is in $\mathsf W$, and therefore $U_{\K}(\mathsf{Post}_{F_{\K}\mathsf Z})\subseteq \W$.
\end{proof}

In the next two subsections we prove that conditions (b) and (c) of Theorem \ref{thm:postnikov} hold as well under the hypotheses of Theorem \ref{thm:bigthm}, thus completing its proof.

\subsection{The first Postnikov factorization}

In the following proposition, which generalizes Lemma 1.15 in \cite{hess:hhg}, we provide conditions under which the adjunction $U_{\K}:\cat M_{\K}\adjunct {}{}\cat M:F_{\K}$ satisfies condition (b) of Theorem \ref{thm:postnikov}.  

Throughout this section, $\C =U_{\K}^{-1}(\mathsf {Cof})$ and $\mathsf W =U_{\K}^{-1}(\mathsf {WE})$.

\begin{prop}\label{prop:1st-factor}  Let $\cat M$ be a model category, and let $\mathsf Z$ be a subset of $\mathsf{Fib}\cap \mathsf {WE}$ such that every morphism $f$ in $\cat M$ admits a factorization $f=qj$, where $q\in \mathsf{Post}_{\mathsf Z}$ and $j\in \mathsf {Cof}$.

If  $\K$ is a comonad on $\cat M$ satisfying axioms (K0) through (K3),
then every morphism $f$ in $\cat M_{\K}$ admits a factorization $f=qj$, where $q\in \mathsf{Post}_{F_{\K}\mathsf Z}$ and $j\in \mathsf {C}$.

\end{prop}

\begin{proof} Let $e$ denote a terminal object in $\cat M$.  Observe that since $F_{\K}$ is a right adjoint, $F_{\K}e$ is a terminal object in $\cat M_{\K}$.

Let $f:(C,\gamma)\to (D, \delta)$ be a morphism of $\K$-coalgebras.  Let 
$$\xymatrix{U_{\K}(C,\gamma)=C\ar[dr]_{j'}\ar [rr]^(0.6){!}&&e\\ &Q\ar [ur]_{q'}}$$
be a factorization of the unique map in $\cat M$ from $C$ to $e$ with $j'\in \mathsf{Cof}$ and $q'\in \mathsf{Post}_{\mathsf Z}$, the existence of which is guaranteed by the hypothesis on $\mathsf Z$.

Taking transposes, we obtain a commuting diagram
$$\xymatrix{(C,\gamma)\ar[dr]_{(j')^{\#}}\ar [rr]^{!}&&F_{\K}e.\\ &F_{\K}Q\ar [ur]_{F_{\K}q'}}$$
Since $F_{\K}$ is a right adjoint and therefore preserves pullbacks and inverse limits,
 $$F_{\K} (\mathsf{Post}_{\mathsf Z})\subseteq \mathsf{Post}_{F_{\K}\mathsf Z},$$
 whence $F_{\K}q'\in \mathsf{Post}_{F_{\K}\mathsf Z}$.  Moreover, the morphism $(j')^{\#}$ admits a factorization
 $$\xymatrix{(C,\gamma)\ar [dr]_{\gamma}\ar [rr]^{(j')^{\#}}&& F_{\K}Q\\&F_{\K}C \ar[ur]_{F_{\K}j'}}$$
 where $\gamma\in \mathsf C$ by (K1) and $F_{\K}j'\in \mathsf C$ by (K2), whence $(j')^{\#}\in \mathsf C$.
 
Axiom (K3) now implies that 
 $$j:=\big( (j')^{\#}, f\big): (C,\gamma) \to F_{\K}Q \times (D,\delta)$$
 is in $\mathsf C$, where the existence of the product $F_{\K}Q \times (D,\delta)$ is guaranteed by (K0).  Furthermore, the projection map
 $$q: F_{\K}Q \times (D,\delta) \to (D, \delta)$$
 is in $\mathsf{Post}_{F_{\K}\mathsf Z}$, as it is the pullback over the unique morphism $(D,\delta)\to F_{\K}e$ of $F_{\K}q'$, and $\mathsf{Post}_{F_{\K}\mathsf Z}$ is closed under pullbacks.  Since $f=qj$, the proof is complete.
\end{proof}

\begin{cor} Under the hypotheses of Theorem \ref{thm:bigthm}, condition (b) of Theorem \ref{thm:postnikov} is satisfied.
\end{cor}

\subsection{The second Postnikov factorization}

We give an inductive proof of condition (c) in Theorem \ref{thm:postnikov} for the category of coalgebras over a comonad. Our proof, which generalizes that of Lemmas 2.13 and 2.14 in \cite{hess:hhg}, can be viewed as dualizing the usual construction of semi-free models of dg-modules over a dg-algebra by recursive attachment of generators, killing all the extra homology in degree $n$ at the $n^{\text{th}}$-stage of the  process.  In essence, to construct a factorization of a morphism of $\K$-coalgebras as a trivial cofibration followed by a fibration, we recursively ``twist in cogenerators'' to kill homotopy of increasingly higher degree, where ``degree'' should be interpreted with respect to a given filtration of weak equivalences. 

Throughout this section, $\C =U_{\K}^{-1}(\mathsf {Cof})$ and $\mathsf W =U_{\K}^{-1}(\mathsf {WE})$.

\begin{prop}  Under the hypotheses of Theorem \ref{thm:bigthm}, 
every morphism of $\K$-coalgebras $f:(C,\gamma)\to (D, \delta)$ admits a factorization $f=pi$, where $i\in \C\cap \mathsf {W}$ and $p\in \mathsf {Post}_{F_{\K}\mathsf X}$.
\end{prop}

\begin{proof} We first establish the base of the induction: $f$ admits a factorization $p_{-1}i_{-1}$, where $i_{-1}\in U_{\K}^{-1}(\mathsf{Cof}\cap \W_{-1})$ and $p_{-1}\in \mathsf {Post}_{F_{\K}\mathsf X}$.  Recall that $\W_{-1}=\mor \cat M$, whence $\mathsf{Cof}\cap \W_{{-1}}=\mathsf{Cof}$.

Let 
$$\xymatrix{U_{\K}(C,\gamma)=C\ar[dr]_{j'}\ar [rr]^(0.6){!}&&e\\ &Q\ar [ur]_{q'}}$$
be a factorization of the unique map in $\cat M$ from $C$ to $e$ with $j'\in \mathsf{Cof}\cap \W$ and $q'\in \mathsf{Post}_{\mathsf X}$, the existence of which follows from the hypothesis that $(\mathsf X, \mathsf Z)$ is a Postnikov presentation of $\cat M$.

Taking transposes, we obtain a commuting diagram
$$\xymatrix{(C,\gamma)\ar[dr]_{(j')^{\#}}\ar [rr]^{!}&&F_{\K}e.\\ &F_{\K}Q\ar [ur]_{F_{\K}q'}}$$
Since $F_{\K}$ is a right adjoint and therefore preserves pullbacks and inverse limits,
 $$F_{\K} (\mathsf{Post}_{\mathsf X})\subseteq \mathsf{Post}_{F_{\K}\mathsf X},$$
 whence $F_{\K}q'\in \mathsf{Post}_{F_{\K}\mathsf X}$.  Moreover, just as in the proof of Proposition \ref{prop:1st-factor}, axioms (K1) and (K2) imply that $(j')^{\#}\in \C$, whence, by axiom (K3),
 $$i_{-1}:=\big( (j')^{\#}, f\big): (C,\gamma) \to F_{\K}Q \times (D,\delta)$$
 is  in $\mathsf C$ as well; the product in the target exists by (K0). Also as in the proof of Proposition \ref{prop:1st-factor},  the projection map
 $$p_{-1}: F_{\K}Q \times (D,\delta) \to (D, \delta)$$
 is in $\mathsf{Post}_{F_{\K}\mathsf X}$.  
 
 We now establish the inductive step of our proof: if $$i_{n}:(C,\gamma)\to (C_{n}, \gamma_{n})\in U_{\K}^{-1}(\mathsf{Cof}\cap \W_{n})$$ for some $n\geq -1$, then there exist 
 $$i_{n+1}\in (C,\gamma)\to (C_{n+1},\gamma_{n+1})\quad\text{ and}\quad p_{n+1}: (C_{n+1},\gamma_{n+1}) \to (C_{n}, \gamma_{n})$$ 
 such that $i_{n+1}\in U_{\K}^{-1}(\mathsf{Cof}\cap \W_{n+1})$, $p_{n+1}\in \mathsf{Post}_{F_{\K}\mathsf X}$ and $i_{n}=p_{n+1}i_{n+1}$.    Applying axioms (K0) and (K6) to $i_{n}$, we obtain an $(n+1)$-equivalence
 $$j_{n+1}: (C,\gamma) \to F_{\K}Q_{n}\times_{F_{\K}P_{n}} (C_{n},\gamma_{n}),$$
 where $P_{n}$ is the cofiber of $U_{\K}i_{n}$ (which represents what we want to ``kill'', at least in filtration $n$, by ``twisting in cogenerators''), and $Q_{n}$ is an acyclic ``based path object'' over $P_{n}$. Axiom (K6) tells us essentially that twisting the cofree coalgebra on the ``cogenerator object'' $Q_{n}$ together with $(C_{n}, \gamma_{n})$ over the cofree coalgebra on $P_{n}$ ``kills the homotopy of the cofiber in filtration $n$.''    

Since $\mathsf{Post}_{F_{\K}\mathsf X}$ is closed under pullbacks, the projection 
$$r_{n+1}:F_{\K}Q_{n}\times_{F_{\K}P_{n}} (C_{n},\gamma_{n})\to (C_{n},\gamma_{n})$$ 
is in $\mathsf{Post}_{F_{\K}\mathsf X}$.  We then apply Proposition  \ref{prop:1st-factor} to write $j_{n+1} = q_{n+1}i_{n+1}$ 
with $i_{n+1} : (C, \gamma) \to (C_{n+1}, \gamma_{n+1})$ such that $i_{n+1} \in \C$, and $q_{n+1} \in \mathsf{Post}_{F_{\K}\mathsf Z}$. 
Here $U_{\K}^{-1}(q_{n+1})$ is a weak equivalence by Proposition \ref{prop:cond-a} and $U_{\K}^{-1}(j_{n+1})$ is in $\W_{n+1}$, so $U_{\K}^{-1}(i_{n+1})$ is in $\W_{n+1}$ by (WE1).    Thus, as required,  $i_{n+1} \in U_{\K}^{-1}(\mathsf{Cof}\cap \W_{n+1})$.  Since $\mathsf{Post}_{F_{\K}\mathsf Z} \subseteq \mathsf{Post}_{F_{\K}\mathsf X}$, the composition $r_{n+1}q_{n+1}=p_{n+1}$ is in $\mathsf{Post}_{F_{\K}\mathsf X}$ as required.

We know thus that there exists a commuting diagram in $\cat M_{\K}$
$$\xymatrix{(C,\gamma)\ar [dd]_{f}\ar[ddr]_{i_{-1}}\ar[ddrr]_{i_{0}}\ar[ddrrrr]_{i_{n}}\\
\\
(D,\delta)&(C_{-1},\gamma_{-1})\ar @{->>}[l]^{p_{-1}}&(C_{0},\gamma_{0})\ar @{->>}[l]^{p_{0}}&\cdots\ar @{->>}[l]^{p_{1}}&(C_{n},\gamma_{n})\ar @{->>}[l]^{p_{n}}&\cdots \ar @{->>}[l]^{p_{n+1}},}$$
where $U_{\K}i_{n}\in \W_{n}$ for all $n \geq -1$.  By axiom (WE1) it follows that $U_{\K}p_{n}\in \W_{n-1}$ for all $n\geq 0$. Axiom (K5) then implies that there is some $k$ such that the partial composition of the tower 
$$p_{\infty,n}: \lim_{m} (C_{m},\gamma_{m})\to C_{n}$$ 
satisfies $U_{\K}p_{\infty,n}\in \W_{n-k-1}$ for all $n\geq k$.  

Let $$i_{\infty}=\lim _{m}i_{m}: (C,\gamma) \to \lim_{m} (C_{m},\gamma_{m}).$$  
Since $p_{\infty,n}\circ i_{\infty}= i_{n}$ for all $n$, axiom (WE1) implies that $U_{\K}i_{\infty}=\W_{n-k-2}$ for all $n\geq k$, from which it follows that 
$$U_{\K}i_{\infty}\in \W.$$  
Moreover  the composition 
$$p_{\infty}:\lim_{m}(C_m,\gamma_{m})\to (D,\delta)$$
of the entire tower is in $\mathsf{Post}_{F_{\K}X}$, since $p_{n}\in \mathsf{Post}_{F_{\K}X}$ for all $n$, and $\mathsf{Post}_{F_{\K}X}$ is closed under inverse limits. Finally
$$f=p_{\infty}i_{\infty},$$
as $f=p_{n}i_{n}$ for all $n$.  

If the model category structure on $\cat M$ is injective, then the factorization $f=p_{\infty}i_{\infty}$ is of the desired form.  Indeed, to conclude that $U_{\K}i_{\infty}$ is a monomorphism in $\cat M$, it suffices to know that at least one of the morphisms $i_{n}:C\to C_{n}$ is a monomorphism.  The proof is therefore complete in this case.

More generally, we can apply Proposition \ref{prop:1st-factor} to $i_{\infty}$, obtaining a factorization 
$$\xymatrix{(C,\gamma)\; \ar @{>->}[r]^{i}_{}& (C',\gamma')\ar @{->>}[r]^(0.4){q}& \lim_{n}(C_{n},\gamma_{n})}$$ 
of $i_{\infty}$, where $i\in \C$ and $q\in \mathsf {Post}_{F_{\K}\mathsf Z}$.  Since $\mathsf Z\subseteq \mathsf{Post}_{\mathsf X}$ by hypothesis, $q\in \mathsf {Post}_{F_{\K}\mathsf X}$, while Proposition \ref{prop:cond-a} implies that $U_{\K}q\in \W$.  It follows then from ``two-out-of-three'' for $\W$ that $U_{\K}i\in \W$ as well.  The desired factorization of $f$ is therefore
$$\xymatrix{(C,\gamma)\; \ar @{>->}[r]^{i}_{\sim}& (C',\gamma')\ar @{->>}[rr]^{p_{\infty}\circ q}&& (D,\delta)}.$$
 \end{proof}
 
 \subsection{Proving axiom (K6)}
 
Experience with explicit examples has shown that to prove that axiom (K6) holds for a certain comonad $\K$ on $\cat M$,  it is often easier to break the problem into two parts: proving two extra axioms about $n$-equivalences and a certain ``stability'' (or ``Blakers-Massey'') axiom hold in $\cat M$, then showing that a stronger version of (K4) holds, which implies (K6).  
 
\begin {prop}\label{prop:K6prime} Let $\cat M$ be a model category with Postnikov presentation $(\mathsf X, \mathsf Z)$ such that $\mathsf Z \subseteq \mathsf {Post}_{\mathsf X}$ and with filtered weak equivalences such that the following axioms hold.
\begin{description}
\item[(WE2)] If $$\xymatrix{ X\;\ar [d]\ar@{>->}[r]^{i}& Y\ar [d]\\Z\;\ar@{>->}[r]^j&P}$$
is a pushout diagram in $\cat M$ where $i\in \mathsf{Cof} \cap \W_{n}$, then $j\in \mathsf{Cof} \cap \W_{n}$ as well.
\item[(WE3)] 
For each $n$ there is a class of {\em special $n$-equivalences}, $ \widetilde \W_n$, such that
$$\W_{n+1}\subseteq \widetilde \W_{n} \subseteq \W_{n}$$
and if $f,g$ are composable, $gf\in \W_n$, and $g\in \widetilde\W_n$, then $f\in \W_n.$

\item [(S)] Given a commuting diagram in $\cat M$ 
$$\xymatrix{A\ar [dd]_{i}^{\sim_{n}}\ar [rrr]\ar@{..>}[dr]^k&&&\;B\ar @{>->}[dl]^\sim_{j}\ar [dd]\\
		&\bullet\ar[r]\ar@{->>}[dl]&\bullet \ar @{->>}[dr]_p\\
		C\ar [rrr]&&&P}$$
in which the outer square is a pushout, the lower inscribed square is a pullback, $i\in \mathsf {Cof}\cap \W_{n}$, $j\in \mathsf {Cof}\cap \W$, and $p\in \mathsf {Post}_{\mathsf X}$, the induced map $k$ from $A$ into the pullback is an $(n+1)$-equivalence.
\end{description}  
A comonad $\K$ on $\cat M$ satisfies axioms (K4) and  (K6)  if it satisfies the following condition.
\begin{description}
\item [(K4')] For every $n\geq -1$ and every pullback diagram in $\cat M$
$$\xymatrix{E\times _{B}D \ar @{->>}[d]\ar [r]& E\ar @{->>}[d]^{\sim_{n}}_{p}\\ D=U_{\K}(D,\delta) \ar [r]^(0.65)f& B}$$
where $p\in \mathsf{Post}_{\mathsf X}\cap \W_{n}$, the induced morphism
$$U_{\K}\big(F_{\K}E\times_{F_{\K}B} (D,\delta)\big) \to E\times _{B}D$$
is a special $n+1$-equivalence (i.e., in $\widetilde\W_{n+1}$).
\end{description}
\end{prop}

\begin{rmk} The induced morphism in axiom (K4') is the one obtained by applying  the universal property of pullbacks to  the commuting diagram
$$\xymatrix{U_{\K}\big(F_{\K}E\times_{F_{\K}B} (D,\delta)\big)\ar @{.>} [dr]^{(\ve_{E},\id_{D})}\ar [dd]\ar [rr]&&U_{\K}F_{\K}E\ar '[d][dd]\ar [dr]^{\ve_{E}}\\
&E\times_{B}D\ar[dd]\ar[rr]&&E\ar[dd]^q\\
D\ar [dr]^{=}\ar'[r]^{U_{\K}f^\sharp}[rr]&&U_{\K}F_{\K}B\ar[dr]^{\ve_{B}}\\
&D\ar[rr]^{f}&&B,}$$
where $f^\sharp: (D,\delta) \to F_{\K}B$ is the transpose of $f$.
\end{rmk}

\begin{proof}
First we show that (K4') implies (K4).  If $p \in \mathsf Z$ then $p$ is a trivial fibration and hence so
is  the pullback map $E\times_{B} D \to D$.   Since $p \in \mathsf {Post}_{\mathsf X}  \cap \W_{n}$ for all $n$, (K4') implies that   
$U_{\K}\big(F_{\K}E\times_{F_{\K}B} (D,\delta)\big) \to E\times _{B}D$ is in $\W_{n+1}$ for all $n$
and hence is a weak equivalence.   The composition of these two maps is the weak equivalence required in (K4).

To see that (K4') implies (K6), consider
$$i: (C,\gamma)\to (D, \delta)\in U_{\K}^{-1}\big(\mathsf {Cof} \cap \W_{n})$$ 
for some $n\geq -1$.  Consider the pushout 
$$\xymatrix{C\ar [d]_{i}^{\sim_{n}}\ar[r]&e\ar[d]^{}\\ D\ar [r]^u&P}$$
in $\cat M$, where $e$ is a terminal object.  Axiom (WE2) implies that the map $e\to P$ is an $n$-equivalence.

Since $(\mathsf X, \mathsf Z)$ is a Postnikov presentation, there is a factorization
$$\xymatrix{e\;\ar [rr]^{\sim_{n}}\ar @{>->}[dr]^{j}_{\sim}&& P\\&Q\ar [ur]^q}$$
with $j \in \mathsf {Cof}\cap \W$ and $q\in \mathsf {Post}_{\mathsf X}$.  By axiom (WE1), $q\in \W_{n}$.

We can therefore apply axiom (S) to the diagram
$$\xymatrix{C\ar  [dd]_{i}^{\sim_{n}}\ar [rrr]\ar@{..>}[dr]^k&&&\;e\ar @{>->}[dl]^\sim_{j}\ar [dd]^{\sim_{n}}\\
		&Q\times_{P}D\ar[r]\ar@{->>}[dl]&Q \ar @{->>}[dr]_q\\
		D\ar [rrr]&&&P}$$
and conclude that the induced morphism $k:C\to Q\times_{P}D$ is an $(n+1)$-equivalence.  
Applying axiom (K4') to the pullback diagram
$$\xymatrix{Q\times _{P}D \ar @{->>}[d]\ar [r]& Q\ar @{->>}[d]^{q}_{\sim_{n}}\\ D=U_{\K}(D,\delta) \ar [r]^(0.65)u& P,}$$
we conclude that the natural morphism
$$U_{\K}\big(F_{\K}Q\times_{F_{\K}P}(D,\delta)\big) \to Q\times_{P}D$$
is in $\widetilde\W_{n+1}$.
On the other hand, $k:C \to Q\times _{P}D$ factors as
$$C= U_{\K}(C,\gamma) \to U_{\K}\big(F_{\K}Q \times _{F_{\K}P}(D,\delta)\big)\to Q\times_{P}D,$$
whence axiom (WE3) implies that $ U_{\K}(C,\gamma) \to U_{\K}\big(F_{\K}Q \times _{F_{\K}P}(D,\delta)\big)$ is an $(n+1)$-equivalence as required.  
\end{proof}

\section{Homotopy theory of comodules over corings}
As in section \ref{subsec:dga}, let $R$ be a commutative ring, and let $\ch$ denote the category of nonnegatively graded chain complexes of $R$-modules, endowed with its usual graded tensor product.  Let $A$ be a differential graded (dg) algebra, and $V$ an $A$-coring, i.e., a comonoid in the category of $A$-bimodules.  We then let $\ma$ and $\mav$ denote the categories of right $A$-modules and of right $V$-comodules in the category of right $A$-modules, respectively.

In this section we apply Theorem \ref{thm:bigthm} to establishing the existence of a model category structure on $\mav$ that is left-induced from the ICM structure on $\ma$ (Proposition \ref{prop:moda}), under reasonable conditions on $V$.  We then construct in the next section explicit fibrant replacement functors in $\mav$, under further conditions on $R$, $A$ and $V$.  We end this section with concrete examples of dg $R$-algebras and corings to which our results apply.  

\begin{rmk} The model category structure on $\mav$ studied here plays an important role in establishing an interesting relationship among the notions of Grothendieck descent, Hopf-Galois extensions and Koszul duality \cite{berglund-hess}.
\end{rmk}

\subsection{Existence of the model category structure}

The goal of this section is to prove the following theorem, which generalizes Theorem 2.10 in \cite{hess:hhg}.
\begin{thm}\label{thm:mav} Let $R$ be a semihereditary commutative ring and $A$ an augmented dg $R$-algebra such that $H_{1}A=0$. If $V$ is an $A$-coring that is semifree as a left $A$-module on a generating graded $R$-module $X$ such that 
\begin{enumerate}
\item $H_{0}(R\otimes _{A}V)=R$, $H_{1}(R\otimes _{A}V)=0$, and
\item $X_{n}$ is $R$-free and finitely generated for all $n\geq 0$,
\end{enumerate}
then the category $\mav$ admits a model category structure left-induced from the ICM structure on $\ma$ by the adjunction
$$\mav\adjunct{U}{-\otimes_{A}V} \ma,$$
where $U$ denotes the forgetful functor. 
\end{thm}

\begin{rmk}  For any dg $R$-algebra $A$ and $A$-coring $V$ with comultiplication $\Delta$ and counit $\ve$, it is clear that $\mav=(\ma) _{\K_{V}}$, where  $\K_{V}$ denotes the comonad $(-\otimes_{A}V, -\otimes _{A}\Delta, -\otimes _{A}\ve)$ on $\ma$.  
\end{rmk}

\begin{rmk}  Recall that a commutative ring $R$ is \emph{semihereditary} if every finitely generated ideal of $R$ is projective \cite [Chapter 4]{rotman}.  Examples of  semihereditary rings include semisimple rings, PID's, rings of integers of algebraic number fields and valuation rings. The requirement that $R$ be semihereditary arises from a connectivity argument in the proof of Theorem \ref{thm:mav} for which it is important  that every submodule of a flat $R$-module be flat, which holds for semihereditary rings \cite[Theorem 9.25]{rotman}.  
\end{rmk}

In order to apply Theorem \ref{thm:bigthm} to proving Theorem \ref{thm:mav}, we need an appropriate notion of filtered weak equivalences in $\ma$.

\begin{defn}  For all $n\geq -1$, let $\W_{n}$ be the set of morphisms $f:M\to N$ of right $A$-modules such that $H_{k}f$ is an isomorphism for all $k< n$ and a surjection for $k=n$.  The elements of $\W_{n}$ are called \emph{$n$-equivalences}.  The special $n$-equivalences, $\widetilde \W_{n}$, required in (WE3) are the $n$-equivalences such that $H_{k}f$ is an isomorphism for $k=n$.
\end{defn}

The connectivity arguments we give below require the following elementary property of $n$-equivalences.

\begin{lem}\label{lem:0conn=conn} Let $n\geq 0$. If a chain map $f:Y\to Z$ is an $n$-equivalence and $f_{1}:Y_{1}\to Z_{1}$ is surjective, then $f_{0}:Y_{0}\to Z_{0}$ is surjective as well.
\end{lem}

\begin{proof}    For any chain map $f:Y\to Z$, there is a commuting diagram of short exact sequences
$$\xymatrix{0\ar [r]&d(Y_{1})\ar[d]_{f_{0}} \ar [r] & Y_{0}\ar[d]_{f_{0}}\ar[r]& H_{0}(Y)\ar[d]_{H_{0}f}\ar [r]&0\\
			0\ar [r]&d(Z_{1})\ar [r] & Z_{0}\ar[r]& H_{0}(Z)\ar [r]&0,}$$
where $d$ denotes the differentials on both $Y$ and $Z$.  If $f:Y\to Z$ is an $n$-equivalence of chain complexes for some $n\geq 0$,  then $H_{0}f$ is at least a surjection.  On the other had, if $f_{1}$ is surjective, then the restriction of $f_{0}$ to $d(Y_{1})$ is surjective.  Thus, under the hypothesis of the lemma, both the righthand and the lefthand vertical morphisms in the diagram above are surjections, which implies that the middle morphism is as well.
\end{proof}

Theorem \ref{thm:mav} is a consequence of the sequence of lemmas below.

\begin{lem}\label{lem:wws} If $A$ is any dg $R$-algebra, then axioms (WE1), (WE2), (WE3), and (S) hold in $\ma$, endowed with its ICM structure, the generic Postnikov presentation $(\mathsf{Fib}, \mathsf {Fib} \cap \W)$ and the filtered weak equivalences defined above.
\end{lem}

\begin{proof}  Axioms (WE1) and (WE3) follows easily from the definitions of $\W_{n}$ and $\widetilde\W_{n}$. To prove (WE2), observe that a monomorphism of $A$-modules is an $n$-equivalence if and only if its cokernel is $(n+1)$-connected.  Since cokernels are preserved under pushout, (WE2) holds.

We now prove that a particularly strong version of axiom (S) holds in $\ma$.  Consider a commuting diagram in $\ma$
$$\xymatrix{M\ar [dd]_{i}\ar [rrr]^{f}\ar@{..>}[dr]^k&&&\;N\ar [dl]^\sim_{j}\ar [dd]^{{i'}}\\
		&P\ar[r]^{f''}\ar[dl]&Q\ar [dr]_p\\
		M'\ar [rrr]^{f'}&&&N'}$$
in which the outer square is a pushout, the lower inscribed square is a pullback, $i$ is a monomorphism, $j$ is a monomorphism and a quasi-isomorphism, and $p$ is a surjection.  We show that the induced map $k$ from $M$ into the pullback $P$ is always a quasi-isomorphism. 

We remark first that $k:M\to P$ is a monomorphism, since $i$ is a monomorphism (cf.~proof of Lemma \ref{lem:1st-factor}).  Showing that $k$ is a quasi-isomorphism is therefore equivalent to proving that $P/M$ is acyclic. 

Let $q:Q \to Q/N$ denote the quotient map.  Since $j$ is a quasi-isomorphism, $Q/N$ is acyclic.  We prove that $qf'':P\to Q/N$ induces an isomorphism $P/M\cong Q/N$, implying that $P/M$ is acyclic, as desired. 

It is immediate that $\im k\subseteq \ker qf''$.   Writing 
$$P=\{(x',y)\in M'\times Q\mid f'(x')=p(y)\},$$
we see that if $(x',y)\in \ker qf''$, then there exists $z\in N$ such that $j(z)=y$.  Since $f'(x')=pj(z)=i'(z)$, and $N'$ is the pushout of $f$ and $i$, we conclude that there exists $x\in M$ such that $x'=i(x)$ and $z=f(x)$, whence $k(x)=(x', y)$, i.e., $\ker qf''\subseteq \im k$.  Thus $\ker qf''= \im k$, and so $qf''$ induces an isomorphism $P/M\cong Q/N$.
\end{proof}

\begin{lem}\label{lem:lim-mav}   Under the hypotheses of Theorem \ref{thm:mav}, all limits in $\mav$ are created in $\ma$.  
\end{lem}

\begin{proof} Since $V$ is $A$-semifree, the endofunctor $-\otimes _{A}V$ on $\ma$ preserves kernels and therefore pullbacks as well, as any pullback in $\ma$ can be calculated as a kernel. It follows that pullbacks in $\mav$ are created in $\ma$.

To conclude, we prove that arbitrary products in $\mav$ are also created in $\ma$.  For every $n\geq 0$, let $B_{n}=\{x_{n1},...,x_{nm_{n}}\}$ be an $R$-basis of $X_{n}$.  Let $\{M_{\alpha}\mid \alpha \in \mathcal J\}$ be any set of right $A$-modules.  The natural map
{\small{$$\big(\prod_{\alpha \in \mathcal J}M_{\alpha}\big)\otimes X\cong \big(\prod_{\alpha \in \mathcal J}M_{\alpha}\big)\otimes_{A}V\to \prod_{\alpha \in \mathcal J}(M_{\alpha}\otimes_{A}V)\cong \prod_{\alpha \in \mathcal J}(M_{\alpha}\otimes X): (y_{\alpha})_{\alpha}\otimes x\mapsto (y_{\alpha }\otimes x)_{\alpha}$$}}
admits a noncanonical inverse
$$\prod_{\alpha \in \mathcal J}(M_{\alpha}\otimes_{A}V)\to \big(\prod_{\alpha \in \mathcal J}M_{\alpha}\big)\otimes_{A}V$$
given in degree $n$ by
\begin{align*}
\prod _{\alpha\in \mathcal J}\bigoplus _{k=0}^{n}\bigoplus_{j=1}^{m_{n-k}}(M_{\alpha})_{k}\cdot x_{n-k, j}&\to\bigoplus _{k=0}^{n}\bigoplus_{j=1}^{m_{n-k}}\big( \prod _{\alpha\in \mathcal J}M_{\alpha}\big)_{k}\cdot x_{n-k, j}\\
\Big(\sum _{k=0}^{n}\sum_{j=1}^{m_{n-k}}y_{\alpha,k,j}\cdot x_{n-k,j}\Big)_{\alpha}&\mapsto \sum _{k=0}^{n}\sum_{j=1}^{m_{n-k}}(y_{\alpha,k,j})_{\alpha}\cdot x_{n-k,j}.
\end{align*}
The functor $-\otimes_{A}V$ therefore commutes with products, whence products in $\mav$ are created in $\ma$.
\end{proof}

\begin{lem}\label{lem:k4-mav} Under the hypotheses of Theorem \ref{thm:mav}, axiom (K4') holds for the comonad $\K_{V}$, with respect to its ICM structure, the generic Postnikov presentation  $(\mathsf{Fib}, \mathsf {Fib} \cap \W)$ and the filtered weak equivalences defined above.
\end{lem}

\begin{rmk} It is easy to prove (K4) directly, but we obtain it here as a consequence of (K4'), which we prefer to prove, as it implies (K6) as well.
\end{rmk}

\begin{proof}  Let $(D,\delta)$ be an object in $\mav$, and let $f:D\to B$ be a morphism in $\ma$, inducing a morphism $f^{\sharp}:(D,\delta)\to (B\otimes _{A}V, B\otimes_{A}\Delta)$ in $\mav$. Let $p:E\to B$ be a fibration in the ICM structure and an $n$-equivalence. We treat separately the cases $n=-1$ and $n\geq 0$. 

Consider first the case $n=-1$, i.e., $p$ is any fibration in the ICM structure on $\ma$. Note that condition (1) of Theorem \ref{thm:mav} implies that for all right $A$-modules $M$, the counit $\ve$ induces isomorphisms $(M\otimes_{A}V)_{k}\cong M_{k}$ for $k=0,1$.   The map $(E\otimes _{A}V)\times _{B\otimes _{A}V}D \to E\times _{B} D$ is therefore an isomorphism in degrees 0 and 1, which implies that it induces an isomorphism in homology in degree 0 and therefore is a special $0$-equivalence (In fact, this is a $1$-equivalence, since the isomorphism in degree 1 implies a surjection in homology in degree 1.)

If $n\geq 0$, we argue as follows. The fibrations in the original right-induced model structure on $\ma$ (cf.~proof of Proposition \ref{prop:moda}) are exactly the chain maps that are surjective in positive degrees, which implies that the fibrations in the ICM stucture are also surjective in positive degrees. On the other hand, by Lemma \ref{lem:0conn=conn}, if $n\geq 0$, then an $n$-equivalence is surjective in degree 0 if it is surjective in degree 1.   It follows that if $p:E\to B$ is a fibration in the ICM structure and an $n$-equivalence for some $n\geq 0$, then it is surjective in every degree.  We can then apply a simple exact sequence argument to show that the fiber $F=\ker p$ of $p$ is $(n-1)$-connected, i.e., its homology is 0 through degree $n-1$.

 Since, as seen in the proof of Lemma \ref{lem:lim-mav}, $-\otimes _{A}V$ commutes with limits,  there is a commuting diagram of short exact sequences of $A$-modules
\begin{equation}\label{eqn:ses}
\xymatrix{0\ar [r]&F\otimes_{A}V\ar [d] _{}\ar [r]&(E\otimes_{A}V)\times _{B\otimes _{A}V}D\ar [d] _{}\ar [r]&D\ar@{=} [d]\ar [r]&0\\
0\ar [r]&F\ar [r]&E\times _{B}D\ar[r]&  D\ar [r]&0,}
\end{equation}
where the leftmost and middle vertical maps are induced by $\ve$. To conclude we show that the hypotheses on $X$ imply that the leftmost map in the diagram is a special $(n+1)$-equivalence, whence the middle map is also a special $(n+1)$-equivalence, as desired.

Filtering $F\otimes _{A}V$ by degree in $X$, we obtain a first-quadrant spectral sequence converging to $H_{*}(F\otimes_{A}V)$, with $E_{p,q}^{0}=F_{q}\otimes (R\otimes_{A}V)_{p}$ and 
$$E_{p,q}^{1}=H_{q}(F)\otimes (R\otimes_{A}V)_{p},$$ 
since $X$ is degreewise $R$-free.  
Note that since $R$ is semihereditary, and $R\otimes_{A}V$ is degreewise $R$-free and therefore $R$-flat, the K\"unneth Theorem (in the guise of \cite[Theorem 11.31]{rotman}) can be applied to prove the existence of short exact  sequences
$$0\to H_{q}F \otimes H_{p}(R\otimes _{A}V) \to E^{2}_{p,q}\to \operatorname{Tor}^{R}\big(H_{q}F, H_{p-1}(R\otimes_{A}V)\big)\to 0$$
for all $q\geq 0$, $p\geq 1$, while $E^{2}_{0,q}\cong H_{q}F \otimes H_{0}(R\otimes _{A}V)\cong H_{q}F$ for all $q\geq 0$, since $H_{0}(R\otimes _{A}V)=R$.
The connectivity condition on $F$ therefore implies that the second page of the spectral sequence satisfies
$E^{2}_{p,q}=0$ for all $q<n$.
Consequently, $H_{m}(F\otimes_{A}V)=0$ for all $m< n$, while $H_{n}(F\otimes_{A}V)\cong H_{n}(F)$.  

It remains only to establish the isomorphism $H_{n+1}(F\otimes _{A}V) \to H_{n+1}F$. It follows from the connectivity condition on $F$ that  $E^{2}_{p,n-p+1}\not =0$ only if $p=0$, as $H_{1}(R\otimes _{A}V)= 0$. 
Since $E^{2}_{0,n+1}\cong H_{n+1}F$,   the desired isomorphism holds if no nonzero differential hits $E^{2}_{0,n+1}$.

 The source of the only possible nonzero differential with target   $E^{2}_{0,n+1}$ is $$E^{2}_{2,n}=H_{n}F \otimes H_{2}(R\otimes_{A}V),$$ 
Note since $H_{1}(R\otimes_{A}V) = 0$, there is no Tor-term in $E^{2}_{2,n}$. 
 The differential
$$d^{2}_{2,n}:H_{n}F \otimes H_{2}(R\otimes_{A}V) \to H_{n+1}F$$
is given by $ d^{2}_{2,n}\big([y]\otimes [x]\big) = [y]\cdot [dx]$, where $\cdot$ denotes the induced action of $H_{*}A$ on $H_{*}F$.  Note that for an arbitrary element $x$ in $X$, $dx$ can have a summand in $A$, and it is the class in $H_{1}A$ of this summand  that acts on $[y]$ for $[x]\in H_{2}(R\otimes_{A}V)$.
Since $H_{1}A=0$ by hypothesis, 
we conclude that  $d^{2}_{2,n}=0$ and therefore that the map $H_{n+1}(F\otimes _{A}V) \to H_{n+1}F$ is an isomorphism. The leftmost vertical map in diagram (\ref{eqn:ses}) is therefore a special $(n+1)$-equivalence.
\end{proof}

\begin{lem}\label{lem:k5-mav} Under the hypotheses of Theorem \ref{thm:mav}, axiom (K5') holds for the comonad $\K_{V}$ on $\ma$, with respect to its ICM structure, the generic Postnikov presentation  $(\mathsf{Fib}, \mathsf {Fib} \cap \W)$ and the filtered weak equivalences defined above.
\end{lem}

\begin{proof} Note first that axiom (K5') holds trivially for $n=-1$.  Consider a tower
$$\cdots \xrightarrow{p_{k+2}} X_{k+1}\xrightarrow{p_{k+1}}X_{k}\xrightarrow{p_{k}}\cdots \xrightarrow{p_{2}} X_{1}\xrightarrow{p_{1}}X_{0}$$
with each $p_{k}$ in $\mathsf{Post}_{\mathsf {Fib}\otimes_{A}V}\cap \W_{n}$, where $\mathsf{Fib}$ denotes the class of fibrations in the ICM structure on $\ma$, and $n\geq 0$.
Each $p_{k}:X_{k}\to X_{k-1}$ is the composition of a tower of length $\lambda$ for some ordinal $\lambda$
$$\cdots \xrightarrow{p_{k,\beta+2}} X_{k,\beta+1}\xrightarrow{p_{k,\beta+1}}X_{k,\lambda}\xrightarrow{p_{k,\beta}}\cdots \xrightarrow{p_{k,2}} X_{k,1}\xrightarrow{p_{k,1}}X_{k,0}=X_{k-1},$$
where  there is a fibration $q_{\beta}:E_{\beta}\to B_{\beta}$ in the ICM structure on $\ma$ and a pullback in $\ma$,
$$\xymatrix{ X_{k,\beta}\ar [d]_{p_{k,\beta}}\ar[r]& E_{\beta}\otimes _{A}V\ar [d]^{q_{\beta}\otimes_{A}V}\\
			X_{k,\beta-1}\ar [r]&B_{\beta}\otimes_{A}V}$$
for every ordinal $\beta < \lambda $.			

Since $q_{\beta}$ is surjective in positive degrees, so are $q_{\beta}\otimes_{A}V$ and thus $p_{k,\beta}$ as well, for all $k$ and $\beta$. Lemma 3.5.3 in \cite{weibel}, which generalizes easily to higher ordinals, therefore implies that each $p_{k}$ is surjective in positive degrees.  By Lemma \ref{lem:0conn=conn}, since $p_{k}$ is also an $n$-equivalence for some $n\geq 0$, it is surjective in degree 0 as well.  

It follows now from Theorem 3.5.8 in \cite {weibel} that there are isomorphisms
$$H_{m}(\lim_{k} X_{k}) \xrightarrow \cong \lim _{k}H_{m}(X_{k})\xrightarrow\cong H_{m}(X_{0})$$
for all $m<n$,  since $H_{m}p_{k}$ is an isomorphism for all $k$, and surjections
$$H_{n}(\lim_{k} X_{k}) \to  \lim _{k}H_{n}(X_{k})\to  H_{n}(X_{0}),$$
since $H_{n}p_{k}$ is a surjection for all $k$. In other words, the composition
$$\lim _{k}X_{k}\to X_{0}$$
is an $n$-equivalence.

Note that we have proved a strong version of (K5'), as the degree of equivalence of the composition is the same as the degree of equivalence of each morphism in the tower.
\end{proof}

\begin{proof}[Proof of Theorem \ref{thm:mav}]  Since $V$ is $A$-semifree, the functor $-\otimes _{A}V$ preserves monomorphisms.  Lemma \ref{lem:1st-factor} implies therefore that axioms (K0) through (K3) hold for the comonad $\K_{V}$.  Lemmas \ref{lem:wws} and \ref{lem:k4-mav} together imply that axioms (K4) and (K6) hold for $\K_{V}$, by Proposition \ref{prop:K6prime}.  Finally, axiom (K5) for $\K _{V}$ follows from Lemma \ref{lem:k5-mav}, as explained in Remark \ref{rmk:K5'}.  
The comonad $\K_{V}$ is therefore tractable and allows inductive arguments, so we can apply Theorem \ref{thm:bigthm} to conclude.
\end{proof}

\begin{exs}\label{exs} The two corings of greatest interest in the context of Hopf-Galois extensions \cite{hess:hhg} and Grothendieck descent \cite{hess:descent} both satisfy the hypotheses of Theorem \ref{thm:mav}, under reasonable conditions.
\begin{enumerate}
\item If $A$ is an augmented, dg $R$-algebra such that $H_{1}A=0$, and $K$ is a dg Hopf algebra such that $H_{0}K_{0}=R$, $H_{1}K_{1}=0$,  and $K_{n}$ is $R$-free and finitely generated for all $n$, then the coring $A\otimes K$ \cite{hess:hhg} satisfies the hypotheses of Theorem \ref{thm:mav}. For example, if $X$ is a 2-reduced simplicial set with finitely many nondegenerate simplices in each degree, and $\Omega$ and $C_{*}$ denote the cobar construction functor and the reduced normalized chain functor, respectively, then $\Om C_{*}X$ is one such dg Hopf algebra \cite{hpst}.
\item Let $B$ and $A$ be augmented dg $R$-algebras such that $A$ is semifree as a left $B$-module on a generating graded $R$-module $Y$ such that $H_{0}(R\otimes _{B}A)=R$, $H_{1}(R\otimes _{B}A)=0$, and $Y_{n}$ is $R$-free and finitely generated for all $n$.  If $H_{1}A=0$, then the canonical coring $A\otimes_{B} A$ associated to the inclusion $B\hookrightarrow A$ \cite{hess:descent} satisfies the hypotheses of Theorem \ref{thm:mav},  as it is left $A$-semifree on $Y$.  

For example, if $B$ is an augmented Hirsch algebra \cite{kadeishvili} such that $B_{0}=R$ and $B_{n}$ is $R$-free and finitely generated for all $n$, then the inclusion of $B$ into the acyclic bar construction $B\otimes _{t_{\op B}} \op B B$ is a multiplicative extension of this type. More generally, if $t: K\to B$ is any Hopf-Hirsch twisting cochain \cite{berglund-hess}, where $H_{0}K=R$, $H_{1}K=0$, and $K_{n}$ is $R$-free and finitely generated for all $n$,  then the multiplicative extension $B\to B\otimes _{t}K$ is also of this type.
\end{enumerate}
\end{exs}

\section{Fibrations  of comodules over corings}

In this section we provide examples of fibrations in the ICM structure on $\ma$ and in the induced structure on $\mav$, where we require that $R$ be a commutative ring that is semihereditary and either Artinian or a Frobenius ring over a field.
In the case of $\mav$, we assume furthermore that the nondifferential algebra $\natural A$ underlying $A$ is a connected (i.e., $A_{0}=R$), \emph{nearly Frobenius} algebra \cite [Definition 2.4]{moore-peterson2}.  In particular,  by  \cite [Theorem 2.7]{moore-peterson2} (see also the remark immediately following the proof),  if $R$ is Artinian or a Frobenius ring over a field, then a graded module over a connected, nearly Frobenius $R$-algebra  is injective if and only if it is projective if and only it is flat.  

 Examples of nearly Frobenius algebras include any algebra underlying a finite dimensional, cocommutative Hopf algebra over $R$, if $R$ is a field  \cite[Section 3]{moore-peterson}.  More generally, the colimit of a filtered, \emph{strongly coherent} diagram of nearly Frobenius algebras is nearly Frobenius  \cite [Definition 2.5, Theorem 2.6]{moore-peterson2}.  In particular,  the mod $p$ Steenrod algebra is nearly Frobenius.     

We need to introduce some helpful notation before stating the main theorem of this section.

\begin{notn} For any dg $R$-algebra $A$, let $T_{A}$ denote the free monoid functor on the category of $A$-bimodules.  In other words, if $M$ is an $A$-bimodule, then
$$T_{A}M= A \oplus \bigoplus _{n\geq 1} M^{\otimes _{A}^{n}},$$
which is naturally a monoid in the category of $A$-bimodules, via concatenation.  Let $y_{1}|\cdots|y_{n}$ denote an arbitrary  element of tensor length $n$.
\end{notn}

\begin{notn} For any $X\in \ob \ch$ with $X_{0}=0$, we let $s^{-1}X$ denote the desuspension of $X$, i.e., $\si X_{n}=X_{n+1}$ for all $n\geq 0$.   Let $\path(X)=(X\oplus \si X, D)$, with $Dx=dx + \si x$ and $D\si x=-\si (dx)$, where $d$ is the differential on $X$.   Let $e_{X}:\path (X) \to X$ denote the natural quotient map.

Note that if $M$ is an $A$-module (respectively, a $V$-comodule in right $A$-modules) such that $M_{0}=0$, then $\si M$ and $\path (M)$ both naturally inherit an $A$-action (respectively, a $V$-coaction and $A$-action) from $M$ such that the projection map $e_{M}$ is a morphism in $\ma$ (respectively, $\mav$), justifying our use of the same notation for this functor on $\ch$, $\ma$ and $\mav$. Observe moreover that 
$$\path (F_{\K_{V}}M)\cong F_{\K_{V}}\path (M)$$
for all right $A$-modules $M$.
\end{notn}

\begin{notn}
If $(M,\delta)\in \ob \mav$, we write $\delta (x)=x_{i}\otimes v^{i}$ (using the Einstein summation convention) for all $x\in M$.  Similarly, for all $v\in V$, we write $\Delta (v)=v_{i}\otimes v^{i}$, where $\Delta $ is a comultiplication on $V$.
\end{notn}

\begin{notn} We apply in this section the Koszul sign convention for commuting elements  of a graded module past each other or for commuting a morphism of graded modules past an element of the source module.  For example,  if $V$ and $W$ are graded algebras and $v\otimes w, v'\otimes w'\in V\otimes W$, then 
$$(v\otimes w)\cdot (v'\otimes w')=(-1)^{mn}vv'\otimes ww',$$
if $v'\in V_{m}$ and $w\in W_{n}$. Furthermore, if $f:V\to V'$ and $g:W\to W'$ are morphisms of graded modules, homogeneous of degrees $p$ and $q$, respectively,  then for all $v\otimes w\in V_{m}\otimes W_{n}$,  
$$(f\otimes g)(v\otimes w)=(-1)^{mq} f(v)\otimes g(w).$$
\end{notn}

\begin{notn} When we need to be especially precise and careful, we use $\natural X$ to denote the graded $R$-module underlying a chain complex $X$.  If it is clear from context, and there is no risk of confusion, then both are denoted $X$, to simplify notation.
\end{notn}

The generalized cobar construction defined below is the tool we need to construct fibrant replacements in $\mav$.  This is no great surprise as, for example,  both the first author in \cite{hess:hhg} and Positselski \cite{positselski} showed that the usual one-sided cobar construction provided fibrant replacements in the category of comodules over a dg coalgebra, at least  over a field and under certain finiteness conditions.  It is nice to see, however, that this useful result generalizes to comodules over corings, even if the proof is  more delicate in the more general case.
 
\begin{defn} Let $A$ be a dg $R$-algebra and $(V,\Delta, \ve, \eta)$ a coaugmented $A$-coring, with coaugmentation coideal $\overline V=\coker (\eta: A \to V)$.  For all $(M,\delta)\in \ob\mav$, let $\Om_{A}(M;V;V)$ denote the object in $\mav$
$$(M\otimes _{A}T_{A}  (s^{-1}  \overline V) \otimes _{A}V, d_{\Om}),$$
where
\begin{align*}
d_{\Om}(x \otimes \si v_{1}|\cdots |\si v_{n}\otimes w)= &\,dx \otimes \si v_{1}|\cdots |\si v_{n}\otimes w \\
&+x\otimes \sum _{j=1}^{n}\pm\si v_{1}|\cdots |\si dv_{j}|\cdots \si v_{n}\otimes w\\
&\pm x \otimes \si v_{1}|\cdots |\si v_{n}\otimes dw\\
&\pm x_{i}\otimes \si v^{i}| \si v_{1}|\cdots |\si v_{n}\otimes w\\
&+x\otimes \sum _{j=1}^{n}\pm\si v_{1}|\cdots |\si v_{j,i}|\si v_{j}^{{i}}|\cdots \si v_{n}\otimes w\\
&\pm x\otimes \si v_{1}|\cdots |\si v_{n}|\si w_{i}\otimes w^{i}
\end{align*}
where all signs are determined by the Koszul rule, the differentials of $M$ and $V$ are both denoted $d$, and both the right $A$-module structure and the $V$-comodule structure are induced from the rightmost copy of $V$.  
\end{defn}

\begin{rmk} Any $A$-coring $V$ that is left $A$-semifree on a generating graded module $X$ satisfying the hypotheses of Theorem \ref{thm:mav} is naturally coaugmented.  Its coaugmentation coideal $\overline V$ is semifree on the generating graded module $\overline X$ such that $\overline X_{0}=0$ and $\overline X_{n}=X_{n}$ for all $n\geq 1$.
\end{rmk}

We can now state precisely how fibrant replacements can be constructed in $\mav$, under strong enough conditions on $R$, $A$ and $V$.

\begin{thm}\label{thm:fib-repl} Let $R$ be a semihereditary commutative ring. Let $A$ be a dg $R$-algebra and $V$ an $A$-coring satisfying the hypotheses of Theorem \ref{thm:mav}. 

If $R$ is also Artinian or a Frobenius ring over a field, $\natural A$ is nearly Frobenius, and the right $A$-action on $V$ satisfies
$$(a\otimes x)\cdot b-(-1)^{mn}ab\otimes x\in A\otimes X_{<m}$$ 
for all $a\in A$, $x\in X_{m}$, $b\in A_{n}$ and $m,n\geq 0$, then 
for all $(M,\delta)\in \ob \mav$ such that $\natural M$ is $\natural A$-free, the coaction map $\delta: M\to M\otimes _{A}V$ factors in $\mav $ as
$$\xymatrix{ M\ar [dr]_{\widetilde \delta}\ar [rr]^{\delta}&&M\otimes_{A}V,\\ 
			&\Om_{A}(M;V;V)\ar [ur]_{p}}$$
where $\widetilde \delta$ is a trivial cofibration and $p$ a fibration, specified by $\widetilde \delta (x)=x_{i}\otimes 1\otimes v^{{i}}$  and $p(x\otimes 1\otimes w)=x\otimes w$, while $p(x \otimes \si v_{1}|\cdots |\si v_{n}\otimes w)=0$ for all $n\geq 1$.    
Moreover, both the source and the target of $p$ are fibrant in $\mav$, whence $\Om_{A}(M;V;V)$ is a fibrant replacement of $M$ in $\mav$.
\end{thm}

As Lemma \ref{lem:free} below shows, we can set $M=V$ in the statement above and obtain in particular a factorization
$$\xymatrix{ V\ar [dr]_{\widetilde \Delta}\ar [rr]^{\Delta}&&V\otimes_{A}V,\\ 
			&\Om_{A}(V;V;V)\ar [ur]_{p}}$$
in $\mav$ with $\widetilde \Delta$ a trivial cofibration and $p$ a fibration between fibrant objects.

\begin{ex} Suppose that $R$ is semihereditary and either Artinian or a Frobenius algebra over a field, e.g., $R$ is a field. Both the Hopf-Galois coring $A\otimes K$ and the descent coring $A\otimes _{B}A$ of Examples \ref{exs} then satisfy the hypotheses of Theorem \ref{thm:fib-repl} if $\natural A$ is nearly Frobenius.  For example, if $R$ is a field, and $\natural A$ underlies a cocommutative graded Hopf algebra over $R$ that is equal to the union of its finite-dimensional sub Hopf algebras, then Theorem \ref{thm:fib-repl} applies.
\end{ex}

Before proving Theorem \ref{thm:fib-repl}, we establish a number of preparatory lemmas and propositions.
In particular, in order to construct fibrant replacements in $\mav$, we need to know something about fibrations and fibrant objects in $\ma$.
 
\begin{prop}\label{prop:fib-ma} Let $R$ be any commutative ring, and let $A$ be a  dg $R$-algebra. If $E$ is a right $A$-module such that $\natural E$ is $\natural A$-injective, then
\begin{enumerate}
\item $E$ is fibrant in the ICM structure on $\ma$, and
\item if $E_{0}=0$, then the projection $e_{E}:\path (E) \to E$ is a fibration in the ICM structure on $\ma$.
\end{enumerate}
\end{prop}

\begin{proof} (1) To show that $E$ is fibrant in $\ma$, we consider an acyclic cofibration 
$i:M\xrightarrow \sim N$
 and a morphism $f:M\to E$ in $\ma$, and prove that $f$ extends over $N$.  Since $i$ is an injection and a quasi-isomorphism, there is a short exact sequence of $A$-module morphisms
$$0\to M\xrightarrow i N\xrightarrow qN/M \to 0,$$ 
with $N/M$ acyclic.  Let $\hom _{A}(-,-)$ denote the natural enrichment of $\ma$ over $\ch$.  The injectivity of $\natural E$ implies that there is an induced short exact sequence of chain complexes
\begin{equation} \label{eqn:ses2}
0\to \hom _{A}(N/M,E) \xrightarrow {q^{*}} \hom _{A}(N, E)\xrightarrow {i^{*}} \hom _{A}(M,E)\to 0.
\end{equation}

Since $i^{*}$ is surjective, there is a morphism of $\natural A$-modules $f':N\to E$ such that $f'\circ i= f$.  Note that
$$i^{*}(df'-f'd)=df-fd=0,$$
i.e., $df'-f'd\in \ker i^{*}=\im q^{*}$.  There exists therefore a unique $A$-linear morphism $\theta: N/M\to E$, homogeneous of degree $-1$, such that $q^{*}(\theta)=df'-f'd$, whence $q^{*}(d\theta+\theta d)=0$.  Since $q^{*}$ is injective, $d\theta+\theta d=0$, i.e., $\theta$ is a cycle in $\hom _{A}(N/M, E)$, which is acyclic, as $N/M$ is acyclic and $\natural E$ is injective.  It follows that there is an $A$-linear morphism $g:N/M\to E$, homogeneous of degree $0$, such that $dg-gd =\theta$.  Setting $\hat f= f'-gq$, we obtain a chain map of $A$-modules such that $\hat f\circ i =f$.
\medskip 

\noindent (2) The proof of this claim is very similar to that of (1).  Recalling that $\path (E)=(E\oplus \si E, D)$, we see that if
\begin{equation}\label{eqn:lifting}
\xymatrix{ M\ar[r]^(0.4){f}\ar [d]_{i}^{\sim}&\path (E) \ar [d]^{e_{E}}\\ 
		N\ar [r]^{g}& E}
\end{equation}
is a commuting diagram in $\ma$, then there is some $A$-linear morphism $\Upsilon: M\to E$ of degree $+1$ such that  $f=(gi, \si  \Upsilon)$, which implies that $d\Upsilon- \Upsilon d= gi$, i.e., $\Upsilon$ is a contracting homotopy for $gi$.  Solving the lifting problem for the diagram (\ref{eqn:lifting}) is therefore equivalent to establishing the existence of an $A$-linear morphism $\widehat \Upsilon:N\to E$ of degree $+1$ such that $\widehat\Upsilon \circ i= \Upsilon$ and $d\widehat\Upsilon- \widehat\Upsilon d= g$.

To prove that $\widehat \Upsilon$ exists, we proceed as in part (1), applying the short exact sequence (\ref{eqn:ses2}) to prove that some extension of $\Upsilon$ to $N$ exists, then using the acyclicity of $\hom _{A}(N/M,E)$ to correct the differential of the extension.
 \end{proof}
  The next lemma, which follows easily from old work on filtered rings and modules, lies behind the conditions we have imposed on $R$ and $A$, as it implies that, under the hypotheses of Theorem \ref{thm:fib-repl},  the right $A$-module underlying $V$ is fibrant in the ICM structure on $\ma$.

\begin{lem}\label{lem:free}  Let $A$ be graded $R$-algebra.  If $M$ is an $A$-bimodule such that 
\begin{itemize}
\item as a left $A$-module, $M$ is free  on a generating graded module $X$, and
\item the right $A$-action on $M$ satisfies  
$$(a\otimes x)\cdot b-(-1)^{mn}ab\otimes x\in A\otimes X_{<m}$$ 
for all $a\in A$, $x\in X_{m}$, $b\in A_{n}$ and $m,n\geq 0$,
\end{itemize}
then $M$ is filtered-free and therefore free as a right $A$-module.
\end{lem}

\begin{proof} Endow $A$ with an increasing, multiplicative filtration, given by $F^{p}A=A$ for all $p\geq 0$ and $F^{p}A=0$ for all $p<0$.  Filter $M$ as well, by $F^{p}M=A\otimes X_{\leq p}$ for all $p\geq 0$ and $F^{p}M=0$ for all $p<0$.  Note  that the right $A$-action on $M$ induces an $R$-linear map
$$F^{p}M \otimes F^{q}A \to F^{p}M$$
for all $p,q$.  

Let $E^{0}_{*}(A)$ and $E^{0}_{*}(M)$ denote the graded $R$-modules associated to the filtrations above.  It is clear that $E^{0}_{*}(A)$ is a graded $R$-algebra concentrated in degree $0$, while $E^{0}_{*}(M)$ is naturally a \emph{free} graded, right  $E^{0}_{*}(A)$-module, on the generating graded module $X$.  It follows then from  \cite[Appendix: Proposition 2]{sridharan} that $M$ is free as a right $A$-module, on a generating graded module isomorphic to $X$.
\end{proof}

The following consequence of Lemma \ref{lem:free} is crucial in the proof of Theorem \ref{thm:fib-repl}.

\begin{cor}\label{cor:injobj} Under the hypotheses of Theorem \ref{thm:fib-repl}, the right $\natural A$-module  
$$\natural (M\otimes_{A}\overline V^{\otimes_{A}n})$$
is injective for all $n\geq 0$.
\end{cor}

\begin{proof} By Lemma \ref{lem:free}, $\natural \overline V$ is a free right $\natural A$-module, which implies that each $\natural (M\otimes_{A}\overline V^{\otimes_{A}n})$ is also $\natural A$-free on the right, since $M$ is a free right $\natural A$-module.  As we have assumed that $\natural A$ is nearly Frobenius, we can conclude that each $\natural (M\otimes_{A}\overline V^{\otimes_{A}n})$ is $\natural A$-injective. 
\end{proof}

We prove Theorem \ref{thm:fib-repl} inductively, repeatedly applying the following simple observation, the easy proof of which we leave to the reader.  Recall that pullbacks in $\mav$ are computed in $\ma$ (Lemma \ref{lem:lim-mav}).

\begin{lem}\label{lem:pullback}  For any morphism $f:M\to N$ in $\mav$ such that $N_{0}=0$, the pullback of $f$ and $e_{N}:\path (N)\to N$ is $(M\oplus \si N,  D_{f})$, where $D_{f}x=dx +\si f(x)$ and $D_{f}\si y=-\si (dy)$ for all $x\in M$ and $y\in N$, where $d$ denotes the differentials of  both $M$ and $N$.
\end{lem}

\begin{proof}[Proof of Theorem \ref{thm:fib-repl}] Note that any signs not given explicitly in this proof are always consequences of the Koszul rule.  We use $s^{-k}$ to denote the endofunctor on $\ch$ given by $k$ iterations of  $s^{-1}$.

For all $n\geq 0$, consider the right $A$-module
$$B^{n}=\big(s^{-n}(M\otimes _{A}\overline V^{\otimes_{A}n+1}), \beta_{n}\big),$$
where 
$$\beta_{n}s^{-n}=(-1)^{-n}s^{-n}\bigg (d\otimes _{A}\overline V^{\otimes_{A} n+1}+\sum _{j=0}^{n}M\otimes _{A}\overline V^{\otimes_{A} j}\otimes_{A}d\otimes_{A}\overline V^{\otimes_{A} n-j}\bigg ).$$
Corollary \ref{cor:injobj} implies that $\natural B^{n}$ in $\natural A$-injective and therefore, by Proposition \ref{prop:fib-ma}, $B^{n}$ is fibrant and $e_{B^{n}}$ is a fibration in the ICM structure on $\ma$ for all $n$.

To begin the recursive construction of $\Om _{A}(M;V;V)$, let 
$$E^{0}= F_{\K_{V}}M=\bigg (M\otimes_{A}V,d_{\Om}^{(0)}\bigg ),$$
where
$$d_{\Om}^{(0)}=d\otimes_{A}V +M\otimes_{A}d.$$
Let 
$$f^{1}=\delta\otimes _{A}V + M\otimes _{A}\Delta: E^{0} \to F_{\K_{V}}B^{0},$$
where we are implicitly composing with the projection $V\to \overline V$ in the middle factor.
A simple calculation shows that $f^{1}$ is a chain map.  Moreover, it is a morphism of $V$-comodules, as it is a sum of two such.

 According to Lemma \ref{lem:pullback}, the pullback of $f^{1}$ and of $F_{\K_{V}}e_{B^{0}}$ is
$$E^{1}= \bigg( (M\otimes _{A}V) \oplus \si (M\otimes _{A}\overline V\otimes _{A}V), d_{\Om}^{(1)}\bigg),$$
where
\begin{align*}
d_{\Om}^{(1)}(x\otimes w)=&dx\otimes w \pm x\otimes dw\\
					&+\si (x_{i}\otimes v^{i}\otimes w + x\otimes w_{j}\otimes w^{j})
\end{align*}
and 
$$d_{\Om}^{(1)}\si (x\otimes v\otimes w)=-\si (dx\otimes v\otimes w\pm x\otimes dv\otimes w \pm x\otimes v\otimes dw),$$
i.e., 
on $M\otimes_{A}V$,
$$d_{\Om}^{(1)}=d\otimes _{A}V+ M\otimes _{A}d+ \si (\delta\otimes _{A}V + M\otimes _{A}\Delta),$$
while on $M\otimes _{A}\overline V\otimes _{A}V$,
$$d_{\Om}^{(1)}\si = -\si (d\otimes _{A}\overline V\otimes _{A}V+ M\otimes _{A}d\otimes _{A} V+ M\otimes _{A}\overline V\otimes _{A}d).$$
The obvious projection map $p^{1}:E^{1}\to E^{0}$ is a morphism in $\mav$, since it is the map given by pulling $F_{\K_{V}}e_{B^{0}}$ back  along $f^{1}$.

The inductive step of the construction goes as follows.  Suppose that for some $N> 1$, we have constructed
$$E^{n}= \bigg( (M\otimes _{A}V) \oplus \bigoplus_{k=1}^{n} s^{-k} (M\otimes _{A}\overline V^{\otimes_{A} k}\otimes _{A}V), d_{\Om}^{(n)}\bigg),$$
for all $1\leq n <N$, where $d_{\Om}^{(n)}$ is defined so that on $M\otimes _{A}\overline V^{\otimes_{A} k}\otimes _{A}V$,
\begin{align*}
d_{\Om}^{(n)}s^{-k}=&(-1)^{-k}s^{-k}\big (d\otimes _{A}\overline V^{\otimes_{A} k}\otimes _{A }V+\sum _{j=0}^{k-1}M\otimes _{A}\overline V^{\otimes_{A} j}\otimes_{A}d\otimes_{A}\overline V^{\otimes_{A} k-j-1}\otimes _{A}V\\
&\quad \qquad \qquad +M\otimes _{A}\overline V^{\otimes_{A} k}\otimes_{A}d\big )\\
&+s^{-(k+1)}\big (\delta \otimes _{A}\overline V^{\otimes_{A} k}\otimes _{A }V+ \sum _{j=0}^{k-1}M\otimes _{A}\overline V^{\otimes_{A} j}\otimes_{A}\Delta\otimes _{A}\overline V^{\otimes_{A} k-j-1}\otimes _{A}V\\
&\quad \qquad \qquad +M\otimes _{A}\overline V^{\otimes_{A} k}\otimes_{A}\Delta \big )
\end{align*}
for all $0\leq k <n$, while on $M\otimes _{A}\overline V^{\otimes_{A} n}\otimes _{A}V$,
\begin{align*}
d_{\Om}^{(n)}s^{-n}=&(-1)^{-n}s^{-n}\big (d\otimes _{A}\overline V^{\otimes_{A} n}\otimes _{A }V+\sum _{j=0}^{n-1}M\otimes _{A}\overline V^{\otimes_{A} j}\otimes_{A}d\otimes_{A}\overline V^{\otimes_{A} n-j-1}\otimes _{A}V\\
&\quad \qquad \qquad +M\otimes _{A}\overline V^{\otimes_{A} n}\otimes_{A}d\big ),
\end{align*}
where we are implicity composing with the projection $V\to \overline V$ in the middle factors.
The obvious projection maps $p^{n}:E^{n}\to E^{n-1}$ are clearly morphisms in $\mav$, for all $n<N$.

Define $f^{N}:E^{N-1}\to F_{\K_{V}}B^{N-1}$ so that $f^{N}s^{-k}=0$ on $M\otimes _{A}\overline V^{\otimes_{A} k}\otimes _{A}V$ for all $k<N-1$, while on $M\otimes _{A}\overline V^{\otimes_{A} N-1}\otimes _{A}V$
\begin{align*}
f^{N}s^{-N+1}=&s^{-N+1}\big (\delta \otimes _{A}\overline V^{\otimes_{A} k}\otimes _{A }V+ \sum _{j=0}^{N-1}M\otimes _{A}\overline V^{\otimes_{A} j}\otimes_{A}\Delta\otimes _{A}\overline V^{\otimes_{A} N-j-2}\otimes _{A}V\\
&\quad \qquad \qquad +M\otimes _{A}\overline V^{\otimes_{A} N-1}\otimes_{A}\Delta \big ),
\end{align*}
where we are implicitly composing with the projection $V\to \overline V$ in the middle factors, as usual.  As in the case of $f^{1}$, it is easy to see that $f^{N}$ is a morphism of $V$-comodules. It is also a chain map, since $\Delta$ is coassociative and $(\delta\otimes_{A}V)\delta= (M\otimes _{A}\Delta)\delta$.

Let $E^{N}$ denote the pullback of $f^{N}$ and $F_{\K_{V}}e_{B^{N-1}}$. By Lemma \ref{lem:pullback},
$$E^{N}= \bigg( (M\otimes _{A}V) \oplus \bigoplus_{k=1}^{N} s^{-k} (M\otimes _{A}\overline V^{\otimes_{A} k}\otimes _{A}V), d_{\Om}^{(N)}\bigg),$$
where the differential $d_{\Om}^{(N)}$ satisfies equations analogous to those satisfied by $d_{\Om}^{(n)}$ for all $n<N$. Moreover, the obvious projection map $p^{N}:E^{N}\to E^{N-1}$, which comes from the pullback, is a morphism in $\mav$.

Let $\mathsf{Fib}$ denote the class of  fibrations in the ICM structure on $\ma$.  Since every $\natural B^{n}$ is $\natural A$-injective,  $e_{B^{n}}\in \mathsf{Fib}$ for all $n\geq 0$, by Proposition \ref{prop:fib-ma}. We have therefore constructed a tower in $\mav$
$$\cdots \to E^{n}\xrightarrow {p^{n}} E^{n-1}\to \cdots \to E^{1}\xrightarrow {p^{0}} E^{0},$$
where each $p^{n}$ is obtained by pulling back a morphism in $F_{\K_{V}}\big (\mathsf{Fib}\big)$, whence the composition of the tower
$$\lim_{n} E^{{n}}\to E^{{0}}$$
is in $\mathsf{Post}_{F_{\K_{V}}(\mathsf{Fib})}$ and is therefore a fibration in the induced model structure on $\mav$.

To conclude we show that $\lim_{n}E^{n}=\Om_{A}(M;V;V)$.  Observe that 
$$\natural E^{n}\cong (M\otimes _{A}V) \oplus \bigoplus_{k=1}^{n} M\otimes _{A}(\si\overline V)^{\otimes_{A} k}\otimes _{A}V.$$
Let $q^{n}:\Om_{A}(M;V;V)\to E^{n}$
denote the obvious quotient map, which is easily seen to be a chain map that respects both the right $A$-action and the right $V$-coaction. Moreover, $p^{n}q^{n}=q^{n-1}$ for all $n$.  It remains therefore only to show that $\Om_{A}(M;V;V)$, endowed with the maps $q^{n}$, satisfies the desired universal property.

Let $N\in \ob \mav$, and let $\{g^{n}:N\to E^{n}\mid n\geq 0\}$ be a set of morphisms in $\mav$ such that $p^{n}g^{n}=g^{n-1}$ for all $n\geq 1$.  Note that $(M\otimes _{A}(\si\overline V)^{\otimes_{A} k}\otimes _{A}V)_{j}=0$ for all $j<k$ and for all $k$, since $(\si \overline V)_{0}=0$ by hypothesis.  It follows that if $y\in N_{n}$, then $g_{n+k}(y)=g_{n}(y)$ for all $k\geq 0$. We can therefore define $g:N\to \Om_{A}(M;V;V)$ by 
$$y\in N_{n}\Longrightarrow g(y)=g_{n}(y),$$
obtaining thus a morphism in $\mav$ such that $q^{n}g=g^{n}$, which is clearly unique.
\end{proof}

\bibliographystyle{amsplain}
\bibliography{cmc}

 \end{document}